\documentclass{amsart}
\usepackage{amssymb,amsmath,exscale,latexsym,amsthm,graphics,overpic}
\usepackage{epsfig}    
\usepackage{epic}      
\usepackage{eepic}     
\usepackage[matrix,arrow,curve,cmtip]{xy} 
\usepackage{enumerate} 
\usepackage{paralist}

\makeatletter
\newtheorem*{rep@theorem}{\rep@title}
\newcommand{\newreptheorem}[2]{%
\newenvironment{rep#1}[1]{%
 \def\rep@title{#2 \ref{##1}}%
 \begin{rep@theorem}}%
 {\end{rep@theorem}}}
\newtheorem*{rep@prop}{\rep@title}
\newcommand{\newrepprop}[2]{%
\newenvironment{rep#1}[1]{%
 \def\rep@title{#2 \ref{##1}}%
 \begin{rep@prop}}%
 {\end{rep@prop}}}
\makeatother

\theoremstyle{plain}
        \newtheorem{theorem}{Theorem}[section]
        \newreptheorem{theorem}{Theorem}
        \newtheorem*{theorem*}{Theorem}
        \newtheorem*{conj*}{Conjecture}
        \newtheorem{lemma}[theorem]{Lemma}
        \newtheorem{cor}[theorem]{Corollary}
        \newtheorem{prop}[theorem]{Proposition}
        \newrepprop{prop}{Proposition}
        
\theoremstyle{definition}
        \newtheorem{definition}[theorem]{Definition}

\theoremstyle{remark}
        \newtheorem{remark}[theorem]{Remark}
        \newtheorem*{remark*}{Remark}

        \newtheorem*{claim}{Claim}

        \newtheorem*{case1}{Case 1}
        \newtheorem*{case2}{Case 2}
        \newtheorem*{case3}{Case 3}
        \newtheorem*{case4}{Case 4}

        \newtheorem{example}[theorem]{Example}
        
\numberwithin{equation}{section}

\providecommand{\defn}[1]{\emph{#1}}
 
\newcounter{mylistnum}

\newcommand{\dist}{\operatorname{dist}}

\newcommand{\clos}  {\operatorname{clos}}
\newcommand{\inte}  {\operatorname{int}}

\newcommand{\id} {\operatorname{id}}
\newcommand{\R}{\mathbb{R}}      
\newcommand{\C}{\mathbb{C}}      
\newcommand{\N}{\mathbb{N}}      
\newcommand{\Z}{\mathbb{Z}}      
\newcommand{\Q}{\mathbb{Q}}      
\newcommand{\Ho}{\mathbb{H}^+}   
\newcommand{\CDach}{\widehat{\mathbb{C}}}
\newcommand{\RDach}{\widehat{\mathbb{R}}}
\newcommand{\D}{\mathbb{D}}      
\newcommand{\Dbar}{\overline{\mathbb{D}}}      

\providecommand{\abs}[1]{\lvert#1\rvert}

\newcommand{\post}{\operatorname{post}}

\newcommand{\CC}{\mathcal{C}}

\newcommand{\EC}{\mathcal{E}}

\newcommand{\E} {\mathbf{E}}

\newcommand{\V} {\mathbf{V}}

\newcommand{\W} {\mathbf{W}}

\newcommand{\A} {\mathbf{A}}

\newcommand{\J}{\mathcal{J}}
\newcommand{\K}{\mathcal{K}}

\newcommand{\wt}{{\tt w}}
\newcommand{\bt}{{\tt b}}

\def\mate{\hskip 1pt \bot \hskip -5pt \bot \hskip 1pt}

\begin{document}

\title{Unmating of rational maps,\\ sufficient criteria and examples}

\date{\today} 

\author{Daniel Meyer}

\address{Daniel Meyer\\
  Jacobs University \\
  School of Engineering and Science\\
  Campus Ring~1\\
  28759 Bremen\\
  Germany 
}

\begin{abstract}
  Douady and Hubbard introduced the operation of mating of
  polynomials. This identifies two filled Julia sets and the dynamics
  on them via external
  rays. In many cases one obtains a rational map. Here the opposite
  question is tackled. Namely we ask when a given (postcritically
  finite) rational map $f$ arises as a mating. A sufficient condition
  when this is possible is given. If this condition is satisfied, we
  present a simple explicit algorithm to unmate the rational map. This
  means we decompose $f$ into polynomials, that when mated yield
  $f$. Several examples of unmatings are presented.
\end{abstract}

\maketitle

\section{Introduction}
\label{sec:introduction}

Douady and Hubbard observed that often ``one can find'' Julia sets of
polynomials within Julia sets of rational maps. This prompted them to
introduce the operation of \emph{mating of polynomials}, see
\cite{MR728980}.  
It is a way to glue together two (connected and
locally connected) filled Julia sets along their boundaries. The
dynamics then descends to the quotient. Somewhat surprisingly one
often obtains a map that is topologically conjugate to a rational
map. 

Thurston's celebrated theorem on the classification of rational maps
among (postcritically finite) topological rational maps (see
\cite{DouHubThurs}) was largely inspired to give a complete answer
when the mating of two polynomials results in (i.e., is topologically
conjugate to) a rational map. 

\smallskip
Here the opposite question is considered. Namely we ask whether a
given rational map $f$ is a mating. This means to decide whether $f$
arises as (i.e., is topologically conjugate to) a mating of
polynomials $P_\wt, P_\bt$. 
If the answer is yes, one wishes to \defn{unmate} the map $f$, i.e.,
obtain the polynomials $P_\wt, P_\bt$. 

The author has recently shown that in the case when $f$ is
postcritically finite and the Julia set of
$f$ is the whole Riemann sphere, that every sufficiently high iterate
indeed arises as a mating, see \cite{inv_Peano} and
\cite{exp_quotients}. Indeed the result remains true for
\emph{expanding Thurston maps}. 

\smallskip
The purpose of this paper is twofold. The first is to explain the
methods used in \cite{inv_Peano} and \cite{exp_quotients}. Namely we
present a sufficient condition that $f$ arises as a mating of
polynomials $P_\wt, P_\bt$. A simple algorithm to unmate $f$ is
presented, i.e., to obtain $P_\wt, P_\bt$ in a simple combinatorial
fashion. The presentation is largely expository, we do not attempt to
provide full proofs (which can be found in \cite{inv_Peano} and
\cite{exp_quotients} and are somewhat technical). 

The second purpose is to contrast the two opposite cases of rational
(postcritically finite) rational maps. Namely we contrast hyperbolic
maps, where each critical point is contained in the Fatou set, with
maps where each critical point is contained in the Julia set
(equivalently the Julia set is the whole Riemann sphere $\CDach$). 

The following contrasting theorems are obtained.
\begin{reptheorem}{thm:hyp_post3}
  Let $f\colon \CDach \to \CDach$ be a hyperbolic, postcritically
  finite, rational map, with $\#\post(f)=3$, that is not a
  polynomial. Then $f$ does not arise as a mating.
\end{reptheorem}

\begin{reptheorem}{thm:P3mate}
  Let $f\colon \CDach \to \CDach$ be a postcritically finite rational
  map, with Julia set $\J(f)=\CDach$ and $\#\post(f)=3$. Then $f$ or
  $f^2$ arises as a mating. 
\end{reptheorem}

In the case when the rational map $f$ is hyperbolic (and
postcritically finite) there is a necessary and sufficient condition
that $f$ arises as a mating. Namely $f$ arises as a mating if and only
if $f$ has an \defn{equator} $\EC$. This is a Jordan curve $\EC
\subset \CDach \setminus \post(f)$ such that $f^{-1}(\EC)$ is
orientation-preserving isotopic to $\EC$ rel.\ $\post(f)$. This
theorem seems to be folklore, but does not appear (to the knowledge of
the author) in the literature. We give a proof here for the
convenience of the reader, see Theorem~\ref{thm:mate_hyp}. 

The existence of an equator however is not the right condition for $f$
to arise as a mating in the non-hyperbolic case. There are many
examples of maps, in particularly Latt\`{e}s maps and maps where
$\#\post(f)=3$, which can be shown to have no equator
(Proposition~\ref{prop:d23_equator},
Proposition~\ref{prop:Lattes_equator}, and
Corollary~\ref{cor:hyp_no_equator}). Nevertheless, it is known that
many of these maps do indeed arise as a mating. 

\smallskip
A sufficient condition that a postcritically finite rational map $f$
with Julia set $\J(f)=\CDach$ arises as a mating is the existence of a
\defn{pseudo-equator}. This means there is a Jordan curve $\CC\subset
\CDach$ that \emph{contains} all postcritical points. Furthermore
there is a pseudo-isotopy $H$ rel.\ $\post(f)$ that deforms $\CC$ to
$\CC^1:= f^{-1}(\CC)$ in an \emph{elementary, orientation-preserving}
way. Such a pseudo-equator is not only sufficient for $f$ to arise as
a mating, but allows one to \defn{unmate} $f$ as well. More precisely
from the pseudo-equator (i.e., from the way $\CC$ is deformed to
$\CC^1$) one obtains a matrix. From this matrix one can obtain the
\emph{critical portraits} of two polynomials $P_\wt, P_\bt$. These
critical portraits determine $P_\wt, P_\bt$ uniquely. The map $f$ is
(topologically conjugate to)
the topological mating of $P_\wt, P_\bt$. 

The question arises whether our sufficient condition is in fact
necessary. The answer however turns out to be no as an explicit
example shows. 

\smallskip
The organization of this paper is as follows. In
Section~\ref{sec:background} we review Moore's theorem. While this
theorem is essential in the theory of matings, one usually only
encounters a weak form of the theorem. We present a stronger version
that deserves to be much better known. 

In Section~\ref{sec:mating-polynomials} we review the mating
construction, as well as relevant theorems. 

In Section~\ref{sec:hyperb-rati-maps} we consider hyperbolic rational
maps. We consider \defn{equators}, existence and their connection to matings.

Section~\ref{sec:an-example} introduces an example (it is a Latt\`{e}s
map), which is used in the following as an illustration. 

The sufficient criterion for $f$ to arise as a mating mating from
\cite{inv_Peano} and 
\cite{exp_quotients} for (postcritically finite) rational maps (whose
Julia set is the whole sphere) is presented in
Section~\ref{sec:anoth-suff-crit}. Namely $f$ arises as a mating if
$f$ has a pseudo-equator.

An equivalent formulation of the sufficient criterion is given in
Section~\ref{sec:connections}. Let $\CC\subset \CDach$ be a
Jordan curve with $\post(f)\subset \CC$ we consider the preimages of
the two components of $\CDach \setminus \CC$. We color one component
of $\CDach \setminus \CC$ white, the other black. Each component of
$\CDach \setminus f^{-1}(\CC)$ is a preimage of one component of
$\CDach \setminus \CC$. The closure of one such component is called a
\defn{$1$-tile}. It it colored white/black if is mapped by $f$ to the
white/black component of $\CDach \setminus \CC$ respectively. The
$1$-tiles tile the sphere in a checkerboard fashion. At critical
points several $1$-tiles intersect. At each critical point we assign a
\defn{connection}. This is a formal assignment which $1$-tiles are
``connected'' at this critical point.
 
The existence of a pseudo-equator is equivalent to the existence of
connections at each critical point, such that the white $1$-tiles form
a spanning tree. Furthermore the ``outline'' of this spanning tree
must be orientation-preserving isotopic to the Jordan curve $\CC$. 

\smallskip
We will need a description of polynomials that is adapted to matings,
i.e., a description in terms of external angles. The one that best
suits our needs is the description via \defn{critical portraits},
introduced by Bielefeld-Fisher-Hubbard and Poirier (see
\cite{MR1149891} and \cite{MR2496235}). Essentially one records the
external angles at critical points. Each such critical portrait yields
a (unique up to affine conjugation) monic polynomial. We review the
necessary material in Section~\ref{sec:critical-portraits}. 

In Section~\ref{sec:find-polyn-mating} we present the algorithm to
unmate rational maps. Namely one obtains from a pseudo-equator a
matrix that encodes how one edge in $\CC$ is deformed to several in
$\CC^1:= f^{-1}(\CC)$ by the corresponding pseudo-isotopy. From the
Perron-Frobenius eigenvector one obtains the critical portraits of the
two polynomials into which $f$ unmates. 

In Section~\ref{sec:examples-unmatings} we present several examples of
unmatings of rational maps. 

\smallskip
In Section~\ref{sec:mating-not-arising} we show that the existence of
a pseudo-equator is not necessary for a rational map $f$ to arise as a
mating. Namely we construct a (rational, postcritically finite)
rational map $f$, whose Julia set is the whole sphere. In fact $f$ is
constructed as the (topological) mating of two (quadratic
postcritically finite) polynomials. We then show that $f$ has no
pseudo-equator.   

\smallskip
We end the paper in Section~\ref{sec:open-questions} with several open
questions. 

\subsection{Notation}
\label{sec:notation}

The Riemann sphere is denoted by $\CDach = \C\cup\{\infty\}$, the
$2$-sphere by $S^2$ (which is a topological object, not equipped with
a conformal structure). 
The unit circle is $S^1$, which will often be
identified with $\R/\Z$. 
The circle at infinity is $\{\infty\cdot
e^{2\pi i \theta} \mid \theta \in \R/\Z\}$, the extension of $\C$ to
the circle at infinity is $\overline{\C}= \C \cup \{\infty\cdot
e^{2\pi i \theta} \mid \theta \in \R/\Z\}$. 
The Julia set is $\J$, the filled Julia set
$\K$. 

We will often \emph{color} objects (polynomials, sets) black and
white. White objects will usually be equipped with a subscript $\wt$,
black objects will be equipped with a subscript $\bt$. 

If $f= a_0 + a_n(z-c)^n + \dots$, then the local degree of $f$ at $c$
is $\deg_f(c)=n$. 

\section{Moore's theorem}
\label{sec:background}

Let $X$ be a compact metric space. 
An equivalence relation $\sim$ on $X$ is called \defn{closed} if
$\{(x,y) \in X\times 
X \mid x\sim y\}$ is a closed subset of $X\times X$. Equivalently, for
any two convergent sequences $(x_n), (y_n)$ in $X$ with $x_n\sim y_n$
for all $n\in \N$, it holds $\lim x_n = \lim y_n$.  

A map $f\colon X\to Y$ \emph{induces} an equivalence relation on $X$ as
follows. For all $x,x'\in X$ let
$x\sim x'$ if and only if $f(x)=f(x')$. 

\begin{lemma}
  \label{lem:sim_f}
  Let $X,Y$ be compact metric spaces, and $f\colon X\to Y$ be a
  continuous surjection. Let $\sim$ be the equivalence relation on $X$
  induced by $f$. Then $\sim$ is closed and $X/\!\sim$ is homeomorphic
  to $Y$.  
\end{lemma}

\begin{proof}
  Let $x_n\to x$ and $y_n\to y$ be two convergent sequence in $X$,
  such that $x_n \sim y_n$ for all $n\in \N$. Then
  \begin{equation*}
    f(x)= \lim f(x_n)= \lim f(y_n) = f(y).
  \end{equation*}
  Thus $x\sim y$. This shows that $\sim$ is closed.

  \smallskip
  Define $h\colon X/\!\sim \,\,\to Y$ by $h([x])= f(x)$. Clearly this is
  a well-defined bijective map. 
  Let $\pi\colon X\to X/\sim$ be the quotient map. 
  Consider an open set $V\subset Y$. Then $f^{-1}(V)=
  \pi^{-1}(h^{-1}(V))$ is open. Thus $h^{-1}(V)$ is open in $X/\!\sim$,
  by the definition of the quotient topology. Thus $h$ is continuous,
  thus a homeomorphism. 
\end{proof}

\smallskip
Moore's theorem is of central importance in the theory of
matings. However in the literature on matings one usually only
encounters a weak form of the theorem. 


\begin{definition}
  \label{def:Moore-type}
  An equivalence relation $\sim$ on $S^2$ is called of
  \emph{Moore-type} if 
  \begin{enumerate}[\upshape(1)]
  \item 
    \label{item:Mooretype1}
    $\sim$ is \emph{not trivial}, i.e., there are at least two
    distinct equivalence classes;
  \item 
    \label{item:Mooretype2}
    $\sim$ is \emph{closed};
  \item 
    \label{item:Mooretype3}
    each equivalence class $[x]$ is \emph{connected};
  \item 
    \label{item:Mooretype4}
    no equivalence class \emph{separates} $S^2$, i.e., $S^2\setminus [x]$ is
    {connected} for each equivalence class $[x]$.  
  \end{enumerate}
\end{definition}

The reason for the name ``Moore-type'' is the classical theorem of
Moore that asserts that if $\sim$ is of Moore-type, then the quotient space
$S^2/\!\sim$ is homeomorphic to $S^2$. This statement is however true in a
stronger form, which we present next. 

\begin{definition}
  \label{def:pseudo-isotopy}
  A homotopy $H \colon X\times [0,1]\to X$ is called a
  \emph{pseudo-isotopy} if $H\colon X\times[0,1) \to X$ is an
  isotopy (i.e., $H(\cdot, t)$ is a homeomorphism for all $t\in
  [0,1)$).  
  We will always assume that $H(x,0)= x$ for all $x\in X$.

  Given a set $A\subset S^2$, we call $H$
  a pseudo-isotopy \defn{rel.\ $A$} if $H$ is a homotopy rel.\ $A$,
  i.e., if $H(a,t)=a$ for all $a\in A$, $t\in [0,1]$. We call the map $h:=
  H(\cdot, 1)$ the \defn{end} of the pseudo-isotopy $H$.    
\end{definition}
We interchangeably write $H(\cdot, t)= H_t(\cdot)$ to unclutter notation.

\begin{lemma}
  \label{lem:pseudo-equiv}
  Let $H\colon S^2\times I \to S^2$ be a pseudo-isotopy. We consider
  the \emph{equivalence relation induced by the end of the
    pseudo-isotopy}, i.e., for $x,y\in S^2$
  \begin{equation*}
    x\sim y \quad \Leftrightarrow \quad H_1(x)= H_1(y).
  \end{equation*}
  Then $\sim$ is of Moore-type. 
    
\end{lemma}

\begin{proof}
  Since $S^2$ is not contractable, it follows that $H_1$ is
  surjective. Thus $\sim$ is not trivial, i.e., not all points in
  $S^2$ are equivalent. Lemma~\ref{lem:sim_f} shows that $\sim$ is
  closed. Furthermore $S^2/\!\sim\,= S^2$.

  \smallskip
  Consider an equivalence class $[x]$. Assume it is not
  connected. Then there is a Jordan curve $C$ that separates
  $[x]$. Let $x_0\in [x]$. We can assume that $H$ keeps $x_0$ fixed,
  otherwise we can compose $H$ with an isotopy such that the
  composition keeps $x_0$ fixed. There is a $\delta>0$ such that
  $\dist(H_t(C),x_0)>\delta$ for all $t\in [0,1]$. 

  We call the component of
  $S^2\setminus C$ containing $x_0$ the interior of $C$, the other
  one the exterior of $C$. Consider a point $y_0\in [x]$ in the
  exterior of $C$. Let $y_t:=H_t(y_0)$ for $t\in [0,1]$. By assumption
  $\lim_{t\to 1} y_t = x_0$. Thus there is a $t_0\in [0,1)$ such that
  $\abs{y_{t_0} - x_0}< \delta/2$. It follows that $y_{t_0}$ is in the
  same component of $C_{t_0}:= H_{t_0}(C)$ as $x_0$. This is impossible.  

  \smallskip
  Assume now that there is an equivalence class $[x]$ that separates
  $S^2$, i.e., $S^2 \setminus [x]$ has (at least) two
  components. We already know that each distinct equivalence class
  $[y]$ is connected. Thus it follows that each equivalence class
  $[y]\neq [x]$ is contained in one of the components of $S^2
  \setminus [x]$. It follows that if we remove the point $[x]$ from
  $S^2/\!\sim$ we obtain a disconnected set. This violates the fact that
  $S^2/\!\sim$ is homeomorphic to $S^2$. 
\end{proof}

\begin{theorem}[Moore, 1925]
  \label{thm:Moore}
  Let $\sim$ be an equivalence relation on $S^2$. Then $\sim$ is of
  Moore-type if and only if $\sim$ can be realized as the end of a pseudo-isotopy $H\colon
  S^2\times [0,1]\to S^2$ (i.e., $x\sim y \Leftrightarrow H_1(x)=
  H_1(y)$). 
\end{theorem}
The ``if-direction'' was shown in Lemma~\ref{lem:pseudo-equiv}. 
A proof of the other direction can be found in
\cite[Theorem~25.1 and Theorem~13.4]{MR872468}. From
Lemma~\ref{lem:sim_f} we immediately recover the original form of
Moore's theorem, i.e., that $S^2/\!\sim$ is homeomorphic to $S^2$. 

The following theorem was originally proved by Baer, see \cite{baer27}
and \cite{baer28},
a more modern treatment can be found in \cite[Theorem 2.1]{MR0214087}.

\begin{theorem}
  \label{thm:homot_isotop}
  Let $P\subset S^2$ be a finite set and $\CC,\CC'\colon S^1 \to
  S^2\setminus P$ be Jordan curves that are homotopic in $S^2\setminus
  P$, i.e., there is a homotopy $H\colon S^1\times I\to S^2\setminus
  P$ such that $H(S^1,0) = \CC$ and $H(S^1,1)=\CC'$. Furthermore
  each component of $S^2\setminus \CC$ contains at least a point in
  $P$. 

  Then $\CC,\CC'$ are isotopic rel.\ $P$, i.e., there is an isotopy
  $K\colon S^2\times I \to S^2$ rel.\ $P$ such that $K(\cdot,0)= \id_{S^2}$
  and $K(\CC,1)=\CC'$.
\end{theorem}

\section{Mating of polynomials}
\label{sec:mating-polynomials}

An excellent introduction to matings can be found in
\cite{MilnorMating}. In fact this paper was the main inspiration for
the papers \cite{inv_Peano} and \cite{exp_quotients}. See also
\cite{MR2636558} and \cite{MR1758798}. 

We still give the basic definitions here, but we will be brief.

\smallskip
Recall that a \defn{Thurston map} is an orientation-preserving
branched covering map $f\colon S^2\to S^2$ of degree at least $2$,
that is postcritically 
finite (see \cite{THEbook} for more background). A Thurston map $f$ is 
called a \defn{Thurston polynomial} if there is a totally invariant
point, i.e., a point $\infty\in S^2$ such that $f(\infty) =
f^{-1}(\infty)= \infty$.

\subsection{Topological mating}
\label{sec:topological-mating}

Let $P$ be a monic polynomial (i.e., the coefficient of the leading
term is $1$) with connected and locally connected filled Julia set
$\K$. Let $\phi\colon \CDach\setminus \Dbar\to \CDach\setminus \K$ be
the Riemann map normalized by $\phi(\infty)=\infty$ and
$\phi'(\infty)=\lim_{z\to \infty} \phi(z)/z >0$ (in fact then
$\phi'(\infty)=1$). By Carath\'{e}odory's theorem (see for example
\cite[Theorem 17.14]{Milnor}) $\phi$    
extends continuously to
\begin{equation}
  \label{eq:defCaraSemi-Conjugacy}
  \sigma\colon S^1=\partial \Dbar\to\partial \K=\J, 
\end{equation}
where $\J$ is the \defn{Julia set} of $P$. 
We call the map $\sigma$ the \defn{Carath\'{e}dory semi-conjugacy} of
$\J$. We remind the reader that every postcritically finite polynomial
has connected and locally connected filled Julia set (see for example
\cite[Theorem 19.7]{Milnor}). 

Consider the equivalence relation on $S^1$ induced by the
Carath\'{e}odory semi-conjugacy, namely 
\begin{equation}
  \label{eq:approx_Cara}
  s\sim t \; :\Leftrightarrow \;\sigma(s)= \sigma(t),
\end{equation}
for all $s,t \in S^1$. Lemma~\ref{lem:sim_f} yields that
$S^1/\!\sim$ is homeomorphic to 
$\J$, where the homeomorphism is given by $h\colon S^1/\!\sim\; \to
\J$, $[s] \mapsto \sigma(s)$. B\"{o}ttcher's theorem (see
for example \cite[$\S$ 9]{Milnor}) says that the Riemann map $\phi$
conjugates $z^d$ to the polynomial $P$ on $\CDach\setminus \Dbar
$, where
$d=\deg P$. This means that the following diagram commutes,
\begin{equation}
  \label{eq:CaraCommDia}
  \xymatrix{
    S^1/\!\sim \ar[r]^{z^d/\sim} \ar[d]_{h}
    &
    S^1/\!\sim \ar[d]^{h}
    \\
    \J \ar[r]_P & \J.
  }
\end{equation}




It will be convenient to identify the unit circle $S^1$ with
$\R/\Z$. We still write 
\begin{equation}
  \label{eq:cara_semi_RZ}
  \sigma\colon \R/\Z\to \J=\partial \K
\end{equation}
for the Carath\'{e}odory semi-conjugacy.

\smallskip
Consider two monic polynomials $P_\wt,P_\bt$ (called the \emph{white} and
the \emph{black} polynomial) of the same degree with
connected and locally connected Julia sets. 
Let $\sigma_\wt,\sigma_\bt$ be the Carath\'{e}odory semi-conjugacies of
their Julia sets $\J_\wt,\J_\bt$. 

Glue the filled Julia sets $\K_\wt,\K_\bt$ (of $P_\wt,P_\bt$) together by
identifying $\sigma_\wt(t)\in \partial \K_\wt$ with
$\sigma_\bt(-t)\in \partial \K_\bt$. More precisely,
we consider the disjoint union of $\K_\wt,\K_\bt$, and let
$\K_\wt\mate \K_\bt$ be the quotient obtained from the equivalence
relation generated by $\sigma_\wt(t)\sim \sigma_\bt(-t)$ for all $t\in
\R/\Z$. The minus sign is customary here, though not essential:
identifying $\sigma_\wt(t)$ with $\sigma_\bt(t)$ amounts to the mating of
$P_\wt$ with $\overline{P_\bt(\bar{z})}$.  
The \defn{topological mating of $P_\wt,P_\bt$} is the map
\begin{align*}
  &P_\wt\mate P_\bt \colon \K_\wt\mate \K_\bt \to \K_\wt\mate \K_\bt,
  \intertext{given by}
  &P_\wt\mate P_\bt|_{\K_i}=P_i,
\end{align*}
for $i=\wt,\bt$. It follows from (\ref{eq:CaraCommDia}) that it is well
defined, 
namely that $x_\wt\sim x_\bt \Rightarrow P_\wt(x_\wt)\sim P_\bt(x_\bt)$ (for all
$x_\wt\in \K_\wt, x_\bt\in \K_\bt$). 

\begin{definition}
  \label{def:arising_mating}
  A Thurston map $f\colon S^2\to S^2$ is said to \defn{arise as
    mating} if $f$ is topologically conjugate to $P_\wt\mate P_\bt$, the (topological)
  mating of polynomials $P_\wt,P_\bt$. 

  The Thurston map $g\colon S^2\to S^2$ is \defn{equivalent to a
    mating} if $g$ is Thurston equivalent to a Thurston map $f\colon
  S^2\to S^2$ arising as a mating. 
\end{definition}

One of the most important result in the theory of matings is the
Rees-Shishikura-Tan theorem.

\begin{theorem}[\cite{MR1182664}, \cite{MR1149864},\cite{MR1765095}] 
  \label{thm:matingdeg2}
  Let $P_\wt= z^2+ c_\wt, P_\bt = z^2+ c_\bt$ be two quadratic
  postcritically finite polynomials, where $c_\wt, c_\bt$ are not in
  conjugate limbs of the Mandelbrot set. Then the topological mating
  of $P_\wt, P_\bt$ exists and is topologically conjugate to a
  (postcritically finite) rational map (of degree $2$). 
\end{theorem}
Recall the a \emph{limb} of the Mandelbrot set is a component of the
complement of the main cardioid.

\subsection{Formal mating}
\label{sec:formal-mating}

We now define the \defn{formal mating}. Its main purpose is to break
up the topological mating into several steps. 

We extend $\C$ to the \defn{circle at infinity} $\{\infty \cdot
e^{2\pi i\theta} \mid \theta \in \R/\Z\}$, and let $\overline{\C}= \C
\cup \{\infty \cdot
e^{2\pi i\theta} \mid \theta \in \R/\Z\}$. Every monic polynomial $P$
of degree $d$ extends to $\overline{\C}$ by $P(\infty \cdot e^{2\pi i
    \theta}) = \infty\cdot e^{2\pi i d \theta}$.  

Consider now two monic polynomials $P_\wt\colon \overline{\C}_\wt \to
\overline{\C}_\wt, P_\bt\colon \overline{\C}_\bt \to
\overline{\C}_\bt$ of the same degree $d$. Here $\overline{\C}_\wt=
\overline{\C}_\bt =\overline{\C}$, but we prefer that $P_\wt,
P_\bt$ are defined on disjoint domains. 
On the disjoint union
$\overline{\C}_\wt \sqcup \overline{\C}_\bt$ we define the equivalence
relation $\approx$ by identifying $\infty \cdot e^{2\pi i \theta}\in
\overline{\C}_\wt$ with $\infty\cdot e^{- 2\pi i \theta} \in
\overline{\C}_\bt$. We call the quotient
\begin{equation*}
  \overline{\C}_\wt \uplus \overline{\C}_\bt 
  := \overline{\C}_\wt \sqcup \overline{\C}_\bt/ \approx,
\end{equation*}
the \emph{formal mating sphere}. Clearly this is a topological
sphere. The map 
$$P_\wt\uplus P_\bt\colon \overline{\C}_\wt \uplus \overline{\C}_\bt 
\to
\overline{\C}_\wt \uplus \overline{\C}_\bt$$ 
given by
\begin{equation*}
  P_\wt \uplus P_\bt(x) =
  \begin{cases}
    P_\wt(x),  & x\in \overline{\C}_\wt\\
    P_\bt(x), & x\in \overline{\C}_\bt
  \end{cases}
\end{equation*}
is well defined. It is called the \defn{formal mating} of $P_\wt,
P_\bt$. The set $\{\infty \cdot e^{2\pi i \theta} \mid \theta \in
\R/\Z\}\subset \overline{\C}_\wt \subset \overline{\C}_\wt \uplus
\overline{\C}_\bt$ is called the \defn{equator} of the formal mating
sphere. 

\smallskip
The \emph{ray-equivalence} relation is the equivalence relation $\sim$
on the formal mating sphere generated by external rays. Here a \defn{closed
external ray} $R(\theta)\subset \overline{\C}_\wt$ (where $\theta \in
\R/\Z$) is one that includes its landing point, as well as its point
at infinity, i.e., $\infty\cdot e^{2\pi i \theta} \in R(\theta)$. Then
$\sim$ is the smallest equivalence relation on $\overline{\C}_\wt
\uplus \overline{\C}_\bt$ such that all points in any closed external
ray $R(\theta)\subset \overline{\C}_\wt$, as well as all points in any
closed external ray $R(\theta)\subset \overline{\C}_\bt$, are
equivalent. 

\smallskip
It is easy to show that $P_\wt \uplus P_\bt$ descends to the quotient
$\overline{\C}_\wt \uplus \overline{\C}_\bt/\!\sim\,$. Moreover this
quotient map is topologically conjugate to the topological mating
$P_\wt \mate P_\bt$. 

\smallskip
We will need the following theorem, originally proved by Rees
\cite[\S 1.15]{MR1149864}. A more general version was proved by
Shishikura \cite[Theorem 1.7]{MR1765095}.
\begin{theorem}
  \label{thm:formal_top}
  Let $P_\wt, P_\bt$ be two hyperbolic, postcritically finite, monic,
  polynomials, such that the formal mating is Thurston equivalent to a
  (postcritically finite) rational map $f\colon \CDach \to
  \CDach$. Then the topological mating $P_\wt\mate P_\bt$ is
  topologically conjugate to $f$. 
\end{theorem}

It is well-known that the Julia set of a postcritically finite
rational map is locally connected. Furthermore each point in the Julia
set of a polynomial is the landing point of only finitely many
external rays (this is Douady's lemma, see
\cite[Lemma~4.2]{MR1765095}). Assume the quotient $\C_\wt \uplus
\C_\bt/\sim$, equivalently $\mathcal{K}_\wt \mate \mathcal{K}_\bt$, is
a topological sphere. Then 
each ray equivalence 
class forms a (actually finite) tree. If this tree does not contain a
critical point it is mapped homeomorphically by the formal mating
$P_\wt \uplus P_\bt$. We have thus
the following lemma.

\begin{lemma}
  \label{lem:preimages}
  Assume the ray-equivalence relation $\sim$ on the formal mating sphere is
  given as above. 
  Let $x\in S^1$ (i.e., in the equator of the formal mating sphere)
  and $[x]$ (i.e., the equivalence class of $x$) be not 
  a postcritical point of the formal 
  mating $P_\wt\uplus P_\bt$. Then the $d$
  preimages of $[x]$ by $P_\wt\uplus P_\bt$ are given by $[(x+k)/d]$
  for $k=0,\dots, d-1$. The sets $[(x+k)/d]$ are disjoint.
\end{lemma}





\section{Equators and hyperbolic rational maps}
\label{sec:hyperb-rati-maps}

\begin{definition}[equator]
  Let $f\colon S^2\to S^2$ be a Thurston map, a Jordan curve
  $\EC\subset S^2\setminus \post(f)$ is an \defn{equator} for $f$ if
  the following three conditions are satisfied.
  \begin{enumerate}
  \item 
    \label{item:eqator1}
    $\widetilde{\EC}:= f^{-1}(\EC)$ consists of a single
    component.
    \setcounter{mylistnum}{\value{enumi}}
  \end{enumerate}

  Since $\EC$ contains no postcritical point, (\ref{item:eqator1}) implies
  that $\widetilde{\EC}$ is a Jordan curve (in $S^2\setminus
  f^{-1}(\post(f))\subset S^2\setminus \post(f)$). Furthermore the
  degree of the map $f\colon \widetilde{\EC} \to \EC$ is $d=\deg f$. 
  \begin{enumerate}
    \setcounter{enumi}{\value{mylistnum}}
  \item 
    \label{item:equator2}
    $\widetilde{\EC}$ is isotopic to
    $\EC$ rel.\ $\post(f)$.  
    \setcounter{mylistnum}{\value{enumi}}
  \end{enumerate}
  Fix an orientation of $\EC$. The curve $\widetilde{\EC}$ then
  inherits an orientation in two distinct ways. First the (covering)
  map $f\colon\widetilde{\EC} \to \EC$ induces an orientation on
  $\widetilde{\EC}$. 

  Second, according to (\ref{item:equator2}) let $H_t$ be the isotopy rel.\ $\post(f)$
  that deforms $\EC$ to $\widetilde{\EC}$, i.e., $H_0 = \id_{S^2},
  H_1(\EC)=\widetilde{\EC}$. We choose the orientation on
  $\widetilde{\EC}$ such that (the homeomorphism) $H_1\colon
  \widetilde{\EC} \to \EC$ is
  orientation-preserving.  

  \begin{enumerate}
    \setcounter{enumi}{\value{mylistnum}}
  \item 
    \label{item:equator3}
    $\widetilde{\EC}$ is \defn{orientation-preserving} isotopic to
    $\EC$ rel.\ $\post(f)$.  
  \end{enumerate}
  This means that the two orientations on $\widetilde{\EC}$ as given
  above agree. 
\end{definition}

If there is an equator $\EC$ for the Thurston map $f$, we say that $f$
has an equator. The importance of equators for matings is shown by the
following theorem. 

\begin{theorem}
  \label{thm:mate_hyp}
  A hyperbolic postcritically finite rational map $f$ arises as a mating
  if and only if $f$ has an equator. 
\end{theorem}

We will need some preparation for the proof of this theorem. 
There are however many classes of Thurston maps that do not have an
equator.  
 
\begin{prop}
  \label{prop:d23_equator}
  Let $f\colon S^2\to S^2$ be a Thurston map that is not a Thurston
  polynomial, with $\#\post(f) \leq 3$. Then $f$ does not have an
  equator. 
\end{prop}

\begin{proof}
  No Thurston map with $\#\post(f)=0$ or $\#\post(f)=1$ exists (see
  \cite[Remark 5.5]{THEbook}).

  If $\#\post(f)=2$ then $f$ is Thurston equivalent to $g=z^m\colon
  \CDach \to \CDach$, where $m\in \Z\setminus \{-1,0,1\}$ (see
  \cite[Proposition 6.3]{THEbook}). The case $m\geq 2$ means that $f$
  is a Thurston polynomial. In the case $m\leq -2$ there is no
  equator, since each Jordan curve $\CC\subset \CDach\setminus
  \post(g)= \C\setminus \{0\}$ which separates the postcritical points
  $0,\infty$ is isotopic (rel.\ $\post(g)=\{0,\infty\}$) to a small
  closed loop around $0$. The preimage of such a loop is isotopic
  rel.\ $\{0,\infty\}$ to itself, but not orientation-preserving
  isotopic to itself.  

  \smallskip
  Assume now that $\post(f)$ consists of three points, which are
  denoted for convenience by $0,1,\infty$. Let $\EC\subset
  S^2\setminus \post(f)$ be a Jordan curve, which we assume to be not
  null-homotopic. Then $\EC$ is isotopic in $S^2\setminus \post(f)$ to
  a small loop $\CC$ around a postcritical point, without loss of generality
  around $\infty$. Let $\widetilde{\CC}:= f^{-1}(\CC)$. For each $c\in
  f^{-1}(\infty)$ there is exactly one small loop around $c$ in
  $\widetilde{\CC}$. Thus $\widetilde{\CC}$ is a single component if
  and only if $\infty$ has a single preimage $c$ by $f$. If $c=\infty$
  this means that $f$ is a Thurston polynomial. If $c\neq \infty$,
  this means that $\widetilde{\CC}$ is not isotopic rel.\ $\post(f)$
  to $\CC$. Thus if $f$ is not a Thurston polynomial it follows 
  that $\CC$, hence $\EC$, is not an equator of $f$.  
\end{proof}

\begin{prop}
  \label{prop:Lattes_equator}
  Let $f\colon \CDach \to \CDach$ be a Latt\`{e}s map. Then $f$ does
  not have an equator. 
\end{prop}

\begin{proof}
  A Latt\`{e}s map $f$ has three or four postcritical points. In the
  first case $f$ has no equator by Proposition
  \ref{prop:d23_equator}. 
 
  \smallskip
  In the second case we use the (well-known) explicit description of
  the map $f$, see \cite[Proposition~9.3]{DouHubThurs} and
  \cite[Theorem~3.1, Section~4, and Section~5]{milnor06_lattes}. 

  Namely there is a lattice $\Lambda= \Z\oplus \tau \Z$, where $\tau$
  is in the upper half plane. The group $G$ is the subgroup of
  $\operatorname{Aut}(\C)$ generated by the translations $z\mapsto
  z+1$, $z\mapsto z+\tau$  and the involution $z\mapsto -z$. The
  quotient map $\wp\colon \C \to \C/G$ may be viewed as the
  Weierstra\ss\ $\wp$-function. Finally
  there is a map  $L\colon \C\to \C$, given by $L(z)= az+b$, where
  $a,b\in \C$. The map $f$ then is topologically conjugate to
  $L/G\colon \C/G
  \to \C/G$.   Here
  $\abs{a}^2=\deg f$. Two cases are possible: either $a\in
  \Z\setminus\{-1,0,1\}$ or $a\notin \R$ (indeed, there is a more
  complete description). 

  \smallskip
  Let $\EC\subset S^2\setminus \post(f)$ be an equator for $f$. Since
  $f$ is not a Thurston polynomial it follows that $\EC$ is not
  peripheral, i.e., each component of $S^2\setminus \EC$ contains two
  postcritical points. The curve $\EC$ is isotopic rel.\ $\post(f)$ to
  a ``straight curve'' $\CC$, i.e., every component of
  $\widetilde{\CC}:=\wp^{-1}(\CC)\subset \C$
  is a straight line. Note that if two such curves $\CC, \CC'$ lift to
  lines with distinct slopes, they are not isotopic rel.\ $\post(f)$. 

  \smallskip
  Let $\CC' := f^{-1}(\CC)$. Then the lift by $\wp$ to
  $\C$ satisfies $\widetilde{\CC}':= \wp^{-1}(\CC') =
  L^{-1}(\widetilde{\CC})$.   

  If $a\in \Z\setminus\{-1,0,1\}$ then $\CC$ has $\abs{a}$ preimages
  by $f$, since $L^{-1}(\wp(\CC))/G$ consist of $\abs{a}$ distinct curves.
  In particular $f^{-1}(\CC)$, hence $f^{-1}(\EC)$ does not consist of a single
  component. Thus $\EC$ is not an equator.

  If $a\notin \R$ then $\widetilde{\CC}' = L^{-1}(\widetilde{\CC})$
  consist of parallel lines which intersect the lines in
  $\widetilde{\CC}$ in the angle $\arg a$. In particular each component
  of $\widetilde{\CC}$ is not isotopic rel.\ $\post(f)$ to
  $\CC$. Thus $\CC$, hence $\EC$ is not an equator. 
\end{proof}

The previous result shows the following. While the existence of an
equator is exactly the right condition to check whether a
(postcritically finite) hyperbolic rational map arises as a mating (by
Theorem~\ref{thm:mate_hyp}), it is not the right condition in the case
when $f$ is not hyperbolic. Namely there are many examples of
Latt\`{e}s maps known which arise as a mating (see
\cite{MilnorMating}, \cite{inv_Peano}), yet they do not have an
equator by the previous theorem.

\begin{cor}
  \label{cor:hyp_no_equator}
  A rational map with parabolic orbifold that is not a polynomial does
  not have an equator.   
\end{cor}

\begin{proof}
  A Thurston map $f$ with parabolic orbifold has at most four
  postcritical points. If $\#\post(f)\leq 3$ the result follows from
  Proposition~\ref{prop:d23_equator}. Every rational map with
  parabolic orbifold and $\#\post(f)=4$ is a Latt\`{e}s map, where the
  result follows from Proposition~\ref{prop:Lattes_equator}. 
\end{proof}

\begin{remark}
  It is possible to find obstructed Thurston maps with parabolic
  orbifold that have an equator. Namely the map $z=x+yi \mapsto 2x +
  yi$ descends to the quotient by the Weierstra\ss\
  $\wp$-function. This is an obstructed map with with four
  postcritical points which has an equator. 
\end{remark}

The following theorem appeared first in Wittner's thesis
\cite[Theorem 7.2.1]{MR2636558}, who credits Thurston. Since Wittner's
thesis is not easily available, we provide the proof here.  

\begin{theorem}
  \label{thm:equator_mating}
  Let $f\colon \CDach\to \CDach$ be a postcritically finite rational map,
  that has an equator $\EC\subset S^2\setminus \post(f)$. Then $f$ is
  Thurston equivalent to the formal mating of two polynomials.
\end{theorem}

\begin{proof}
  Let $f$ be a postcritically finite rational map, that has an
  equator. From Corollary~\ref{cor:hyp_no_equator} it follows that $f$
  has hyperbolic orbifold, which we assume from now on. 



  \smallskip
  Let $\EC$ be an equator for $f$, and $\widetilde{\EC}:=
  f^{-1}(\EC)$. Let $\widetilde{U}_\wt,\widetilde{U}_\bt$ be the two
  components of $\CDach\setminus \widetilde{\EC}$. The map $f$ maps
  $\widetilde{U}_\wt$, as well as $\widetilde{U}_\bt$, (properly) to one
  component of $\CDach \setminus \EC$. Thus we denote the two
  components $U_\wt,U_\bt$ of $\CDach \setminus \EC$ such that
  $f(\widetilde{U}_\wt) = U_\wt$, $f(\widetilde{U}_\bt)= U_\bt$.  

  \smallskip
  Let $H\colon \CDach \times [0,1]$ be an isotopy rel.\ $\post(f)$
  that deforms $\EC$ to $\widetilde{\EC}$, i.e., $H_0= \id_{\CDach}$ and
  $H_1(\EC)= \widetilde{\EC}$. Then $\widetilde{\EC}$ is (forward and
  backward) invariant for
  the map $\widetilde{f} \colon
  \CDach \to \CDach$ given by $\widetilde{f}= H_1 \circ f$, i.e.,
  $\widetilde{f}(\widetilde{\EC})=\widetilde{\EC}=
  \widetilde{f}^{-1}(\widetilde{\EC})$. We can choose the 
  isotopy $H$ such that $\widetilde{f}\colon{\widetilde{\EC}} \to
  \widetilde{\EC}$ is 
  topologically conjugate to $z^d\colon S^1 \to S^1$ (where $d=\deg
  \widetilde{f}= \deg f$). We will assume this from now on. 
   
  Since $\widetilde{\EC}$ is
  orientation-preserving isotopic (rel.\ $\post(f)$) to $\EC$, it
  follows that $H_1({U}_\wt)= \widetilde{U}_\wt$ and $H_1({U}_\bt)=
  \widetilde{U}_\bt$.  
  
  \smallskip
  On $\widetilde{X}_\wt:= \widetilde{U}_\wt \cup \widetilde{\EC}$ we
  consider the equivalence relation which has the following
  equivalence classes. The set $\widetilde{\EC}$ is an
  equivalence class and every singleton $\{x\}\subset U_\wt$ is an
  equivalence class. 

  Consider the quotient space $S^2_\wt:=
  \clos\widetilde{U}_\wt/\widetilde{\EC}$. Clearly $S^2_\wt$ is a
  topological sphere.  

  The map $H_1 \circ f$ maps $\widetilde{\EC}$ to $\widetilde{\EC}$
  and each $x\in \widetilde{U}_\wt$  into $\widetilde{U}_\wt$. Thus this
  map descends naturally to the quotient $S^2_\wt$, i.e., to a map
  $p_\wt\colon S^2_\wt\to S^2_\wt$. Call the point $[\widetilde{\EC}]\in S^2_\wt$
  for convenience $\infty$. We note that $p_\wt(\infty)=
  P_\wt^{-1}(\infty)=\infty$. Clearly $p_\wt$ is a postcritically finite branched
  covering. Thus $p_\wt$ is a Thurston polynomial. 

  \smallskip
  If $p_\wt$ would be obstructed it would have a L\'{e}vy cycle. Then
  $f$ would have a L\'{e}vy cycle as well. Since we assume that $f$ is
  a rational map with hyperbolic orbifold this cannot happen. Thus
  $p_\wt$ is Thurston equivalent to a polynomial $P_\wt$, which we assume to
  be monic and centered. We extend this polynomial to the circle at
  infinity $S^1_\infty= \{\infty \cdot e^{2\pi i\theta} \mid \theta
  \in \R/\Z\}$, this extension is still denoted by $P_\wt \colon
  \overline{\C}_\wt\to \overline{\C}_\wt$. 

  \begin{claim}
    There are homeomorphisms $h_0, h_1\colon \overline{\C}_\wt \to
    \widetilde{X}_\wt$ with the following properties.
    \begin{itemize}
    \item $h_0\circ P_\wt(z) = \widetilde{f}\circ h_1(z)$ for all
      $z\in \overline{\C}_\wt$, i.e., the following diagram commutes
      \begin{equation*}
        \xymatrix{
          \overline{\C}_\wt \ar[r]^{h_1} \ar[d]_{P_\wt}
          &
          \widetilde{X}_\wt \ar[d]^{\widetilde{f}}
          \\
          \overline{\C}_\wt \ar[r]_{h_0} & \widetilde{X}_\wt\,;
        }
      \end{equation*}
    \item $h_0,h_1$ are isotopic rel.\ $\post(P_\wt) \cup S^1_\infty$,
      in particular $h_0=h_1$ on $S^1_{\infty}$. 
    \end{itemize}
  \end{claim}

  \begin{proof}[Proof of Claim]
    We will deform $P_\wt\colon \overline{\C}_\wt\to
    \overline{\C}_\wt$ by an isotopy rel.\ $\post(P_\wt)\cup
    S^1_\infty$ and $\widetilde{f}\colon \widetilde{X}_\wt\to
    \widetilde{X}_\wt$ by an isotopy rel.\ $\post(\widetilde{f}) \cup
    \widetilde{\EC}$. It is enough to prove the statement with these
    deformed maps. 

    More precisely, we can postcompose $P_\wt$ with an isotopy rel.\
    $\post(P_\wt)\cup S^1_\infty$ such that the resulting map is
    topologically conjugate to $\varphi\colon S^1 \times [0,1] \to
    S^1\times[0,1]$, $\varphi(\theta, t) = (d\theta,t)$ in an annulus
    $A$ containing $S^1_\infty\subset \overline{\C}_\wt$. We call this
    deformed map $\widetilde{P}_\wt$.
    The boundary component of $A$ distinct from $S^1_\infty$ is called
    $\CC$. Note that $\widetilde{P}_\wt(\CC)=\CC$. 

    \smallskip
    Since $P_\wt$, hence $\widetilde{P}_\wt$, is Thurston equivalent
    to $p_\wt$, it follows that there are homeomorphisms $k_0, k_1\colon
    \C_\wt\to \widetilde{U}_\wt$ such that $k_0\circ \widetilde{P}_\wt
    = \widetilde{f}\circ k_1$ on $\C_\wt$, that are isotopic rel.\
    $\post(P_\wt)$. 

    Let $\widetilde{\CC}:= k_0(\CC)\subset \widetilde{U}_\wt$ and
    $\widetilde{\CC}^1:= \widetilde{f}^{-1}(\CC)= k_1(\CC)$. We can
    precompose $\widetilde{f}$ with a pseudo-isotopy $I$ rel.\
    $\post(\widetilde{f})\cup \widetilde{\EC}$, so that the resulting
    map $\widetilde{f}\circ I_1$ leaves $\widetilde{\CC}$ invariant,
    and $\widetilde{f}\circ I_1$ is topologically conjugate to
    $\varphi$ as above. Let $\widetilde{A}$ be the annulus between
    $\EC$ and $\widetilde{\CC}$. 

    We define $h_0=k_1$ on $\C_\wt \setminus A$ and $h_1= I_1\circ k_1$
    on $\C_\wt\setminus A$. Since both $\widetilde{P}_\wt$ and
    $\widetilde{f}\circ I_1$ are topologically conjugate to $\varphi$
    on $A$ respectively $\widetilde{A}$, we can
    extend $h_0, h_1$ to $A$ (mapping to $\widetilde{A}$) such that
    the claim holds.  
  \end{proof}

  \smallskip
  In the same fashion we define $S^2_\bt:= \clos
  \widetilde{U}_\bt/\widetilde{\EC}$ and the map $p_\bt\colon S^2_\bt\to
  S^2_\bt$. Again this is a Thurston polynomial, which is Thurston
  equivalent to a (monic, centered) polynomial $P_\bt$ (i.e., is not
  obstructed). The analog statement to the claim above holds for
  $P_\bt$. 

  \smallskip
  It is not necessarily true that $\widetilde{f}$ (hence $f$) is
  Thurston equivalent to the formal mating of $P_\wt, P_\bt$. However
  the maps $h_0, h_1$ in the above claim have to map one of the $d-1$
  fixed points of $P_\wt$ on $S^1_\infty\subset \overline{\C}_\wt$ to
  a fixed point of $\widetilde{f}$ on $\EC$. Similarly for
  $P_\bt$. Thus $\widetilde{f}$, hence $f$, is Thurston equivalent to
  the formal mating of $P_\wt$ and $c^{-1}P_\bt(cz)$, where $c=
  e^{2\pi i j/(d-1)}$, for some $j=0, \dots d-2$.


\end{proof}

\begin{proof}[Proof of Theorem~\ref{thm:mate_hyp}]
  Let $f\colon \CDach \to \CDach$ be a hyperbolic postcritically
  finite rational map. 

  \smallskip
  Assume $f$ arises as a mating. Let $S^2= \overline{\C}_\wt\uplus
  \overline{\C}_\bt$ be the formal mating sphere with equator
  $S^1:=\{\infty\cdot e^{2\pi i \theta} \mid \theta \in
  [0,2\pi]\}\subset \overline{\C}_\wt \subset S^2$ and $\sim$ be the
  ray-equivalence on $S^2$. Let $H\colon S^2\times I\to S^2$ be
  the pseudo-isotopy realizing $\sim$ according to Moore's
  Theorem. 

  We fix an appropriate $\epsilon>0$ such that for $t\in [1-\epsilon,1]$
  the curve $H(S^1,t)$ does not meet $\post(f)$, i.e., $H(S^1\times
  [1-\epsilon,1])\subset S^2\setminus\post(f)$. Consider the two
  parametrized curves 
  $\gamma_1,\gamma_{1-\epsilon}\colon S^1\to 
  S^2_f$ given by $\gamma_1(\infty\cdot e^{2\pi i \theta}) = H_1(\infty\cdot e^{2\pi i d
    \theta})$, and more generally $\gamma_{t}(\infty\cdot e^{2\pi i \theta}) =
  H_{t}(\infty\cdot e^{2\pi i d \theta})$ for each $t\in [1-\epsilon,1]$
  where $d=\deg f$. Note that they cover the image (at least) $d$-fold. 
  Clearly
  $H\colon S^1\times[1-\epsilon,1]\to S^2$ is a homotopy deforming
  $\gamma_{1-\epsilon}$ to $\gamma_1$. 
  
  Consider now $\widetilde{\gamma}_1\colon S^1\to S^2$ given by
  $\widetilde{\gamma}_1(\infty\cdot e^{2\pi i \theta}):= H_1(e^{2\pi i
    \theta})$. By \eqref{eq:CaraCommDia} $\widetilde{\gamma}_1$ is a
  lift of $\gamma_1$ by 
  $f$, i.e., $f\circ \widetilde{\gamma}(\infty\cdot e^{2\pi i \theta}) =
  \gamma(\infty\cdot e^{2\pi i \theta})$. By the standard lifting theorem
  of homotopies by covering maps it follows 
  that the homotopy $H\colon S^1 \times [1-\epsilon,1]\to S^2$ can
  be lifted by $f$ to a homotopy $\widetilde{H}\colon S^1\times
  [1-\epsilon,1]\to S^2$ with $\widetilde{H}_1=
  \widetilde{\gamma}_1$, i.e., $f\circ \widetilde{H} = H$. Let
  $\widetilde{\gamma}_t\colon S^1\to S^2$ 
  be given by $\widetilde{\gamma}_t(\infty\cdot e^{2\pi i \theta}):=
  \widetilde{H}_t(\infty\cdot e^{2\pi i \theta})$. Thus $f\circ
  \widetilde{\gamma}_t(\infty\cdot e^{2\pi i \theta}) = \gamma_t(\infty\cdot e^{2\pi
      i \theta})$. Thus
  $\widetilde{\gamma}_t$ is a component of $f^{-1}(\gamma_t)$. 
  
  It remains to show that there is no other component, i.e., that
  $\widetilde{\gamma}_t$ is the whole preimage of $\gamma_t$ by
  $f$. This is seen using Lemma \ref{lem:preimages}. Indeed for any
  $\gamma_1(\infty \cdot e^{2\pi i \theta})= H_1(\infty\cdot e^{2\pi i d
    \theta})$ the points $H_1(\infty\cdot e^{2\pi i \theta +
    \frac{k}{d}})$, $k=0,\dots , d-1$ are the $d$ \emph{distinct}
  preimages by $f$. By continuity $\widetilde{\gamma}_t(\infty\cdot e^{2\pi
    i \theta +\frac{k}{d}})$, $k=0,\dots, d-1$ are distinct for $t$
  sufficiently close to $1$. Thus $\widetilde{\gamma}_t$ is the whole
  preimage of $\gamma_t$ by $f$. 

  Since $\CC_t:={\gamma}_t(S^1)$ contains no postcritical point
  $\widetilde{\CC}_t:=\widetilde{\gamma}_t(S^1)$ is a simple curve. 
  Finally the concatenation of $H$ and $\widetilde{H}$ deforms $\CC_t$
  to $\widetilde{\CC}_t$, thus they are homotopic
  rel. $\post(f)$. Clearly each component of $S^2 \setminus \CC_t$
  contains at least a postcritical point. Thus $\CC_t,
  \widetilde{\CC}_t$ are orientation-preserving isotopic rel.\
  $\post(f)$ by Theorem~\ref{thm:homot_isotop}.  

  \smallskip
  Assume now that $f$ has an equator. Then $f$ is Thurston equivalent
  to the formal mating of two polynomials $P_\wt,P_\bt$ by
  Theorem~\ref{thm:equator_mating}. From the Rees-Shishikura theorem, i.e.,
  Theorem~\ref{thm:formal_top}, it follows that $f$ is topologically
  conjugate to (the topological mating) $P_\wt\mate P_\bt$. 
\end{proof}

Theorem~\ref{thm:mate_hyp} together with
Proposition~\ref{prop:d23_equator} immediately yields the following.

\begin{theorem}
  \label{thm:hyp_post3}
  Let $f\colon \CDach \to \CDach$ be a hyperbolic, postcritically
  finite, rational map, with $\#\post(f)=3$, that is not a
  polynomial. Then $f$ does not arise as a mating.
\end{theorem}

In the case when $f$ is a polynomial it arises trivially as the mating
of itself with $z^d$ (where $d=\deg f$). Note that each iterate $f^n$
has the same postcritical set. Thus in the case of the previous
theorem no iterate $f^n$ arises as mating. This of course is in
contrast to the results from \cite{inv_Peano} and
\cite{exp_quotients}.


  
  

\section{An Example}
\label{sec:an-example}

We provide an example of a rational map to be able to illustrate the
following. We give several descriptions (naturally equivalent) of the
same map.

\smallskip
Consider two equilateral triangles. We glue them together along their
boundary. This yields a topological sphere denoted by $\Delta$. We
color one side (i.e., 
one of the equilateral triangles) white and call it $T_\wt$, the other
one is colored black and called $T_\bt$. We call $T_\wt, T_\bt$ the
\defn{$0$-triangles}. 

Glue again two equilateral triangles (of the same size as before)
together to form a topological sphere $\Delta^1$ as before. Each face
is now divided into four equilateral triangles of half the
side-length. We color these small triangles in a \emph{checkerboard
  pattern} black and white. This means that two small triangles which
share an edge have different color. These small triangles are called
\defn{$1$-triangles}. 

Consider a small white triangle $T_1\subset \Delta^1$. The map $f$ is
given on $T_1$ as follows.
Scale $T_1$ by the factor $2$ and map it to the big white
triangle in $\Delta$. 
Consider now a small black triangle $T_2\subset\Delta^1$ that intersects
$T_1$  in an edge.
The map $f$ can be extended continuously to $T_2$ by scaling by the
factor $2$ and mapping it to the black triangle in $\Delta$. 

\begin{figure}
  \centering
  \begin{overpic}
    [width=11cm, 
    tics =20]{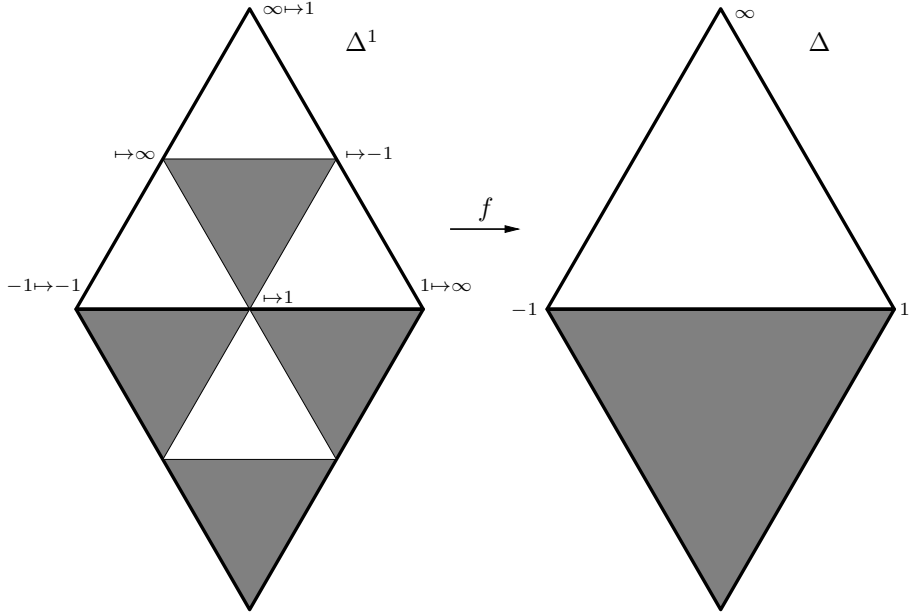}
    \put(89,68){$\Delta$}
    \put(33,68){$\Delta^1$}
    \put(-8,39){$\scriptstyle{-1\mapsto -1}$}
    \put(42,39){$\scriptstyle{1\mapsto \infty}$}
    \put(23,37.5){$\scriptstyle{\mapsto 1}$}
    \put(23,72.5){$\scriptstyle{\infty\mapsto 1}$}
    \put(33,55){$\scriptstyle{\mapsto -1}$}
    \put(5,55){$\scriptstyle{\mapsto \infty}$}
    \put(53,36){$\scriptstyle{-1}$}
    \put(100,36){$\scriptstyle{1}$}
    \put(80,72){$\scriptstyle{\infty}$}
    \put(49,48){$f$}
  \end{overpic}
  \caption{A Latt\`{e}s map with signature $(3,3,3)$.}
  \label{fig:1}
\end{figure}

Continuing in this fashion we obtain a map $\Delta^1\to
\Delta$. Identifying $\Delta^1$ with $\Delta$ we obtain a branched
covering map from a topological sphere to itself. Note that the
critical points (i.e., the points where $f$ is not locally injective)
are the vertices of the small triangles. These are mapped to the
vertices of $\Delta$, each such vertex is mapped to (possibly) another
vertex of $\Delta$. Thus the map is postcritically finite.

\smallskip
The map may not look as a rational map to many readers. To the reader
familiar with Thurston's classification of rational maps among Thurston
maps (see \cite{DouHubThurs}) we remark that $f$ has only $3$
postcritical points, thus it has no Thurston obstruction (the
orbifold of $f$ though is parabolic).

\smallskip
There is however a standard way to view every polyhedral surface as a
Riemann surface (see \cite{Beardon_Riemann}). We outline the
construction in the case at hand.

Consider first a point $p\in \Delta$ that lies in the interior of
either one of the two triangles from which $\Delta$ is built, say
$p\in \inte T_\wt$. Then (any) orientation-preserving, isometric map
from $\inte T_\wt$ to (the interior of) an equilateral triangle in the
plane is a chart. 

Assume now that $p$ lies on a common edge $E$ of $T_\wt$
and $T_\bt$, but is not a vertex of $\Delta$. Then the
map to (the interior of) the union of
two equilateral triangles in the plane is a chart. Here of course we
demand that this chart is orientation preserving, and the interiors of
$T_\wt,T_\bt, E$ are mapped isometrically. 

Finally if $p$ is a vertex of $\Delta$ we note that the (Euclidean)
total angle around $p$ is $2\pi/3$. Thus we can map a small open
neighborhood of $p$ in $\Delta$ essentially by the map $z\mapsto z^3$
to a planar domain. 

Changes of coordinates are conformal, thus we have defined a Riemann
surface. Clearly this is compact and simply connected, thus
conformally equivalent to the Riemann sphere $\CDach$. Note that the
map $f\colon \Delta\to \Delta$ is holomorphic with respect to these
charts. Thus changing the coordinate system yields a holomorphic,
hence a rational map $f\colon \CDach\to \CDach$. 

\smallskip
A slightly different way to construct the map may be thought of as
constructing the uniformizing map above explicitly.  
Namely map the equilateral triangle $T_\wt$ to the upper half-plane
conformally, i.e., by a Riemann map $\varphi$ normalized such that the three
vertices are mapped to $-1,1,\infty$. Recall that $T_\wt$ is
subdivided into four triangles of half the
side-length. Consider one (of the three) such triangles which is
colored white $T_1$. We consider the image in the upper half-plane,
i.e., $T'_1:= \varphi(T_1)\subset \clos\Ho$. 

We map $T'_1$ conformally to the upper half-plane, such that the
images of the vertices by $\varphi$ are mapped to
$-1,1,\infty$. Figure \ref{fig:1} indicates which  vertices are mapped
to which of the points $-1,1,\infty$. 

Consider now the black $1$-triangle $T_2\subset T_\wt$. It shares
an edge $E$ with $T_1$. Its image is $T'_2:= \varphi(T_2)$. 

Note that $T_2$ is the reflection of $T_1$ along $E$. Thus $T'_2$ is
the \defn{conformal reflection} of $T'_1$ along $E':= \varphi(E)$. By
the Schwarz reflection principle it follows that the map $f\colon T'_1
\to \clos \Ho$ extends conformally to $T'_2$ and maps this set to the
lower half-plane (vertices are mapped to $-1,1,\infty$). 

Continuing in this fashion we construct a map $f\colon \clos \Ho=
\varphi(T_\wt)\to \CDach$. There are now two ways to proceed. Either
we map $T_\bt$ conformally to the lower half-plane and proceed as
before. Alternatively we note that the map $f\colon \clos \Ho=
\varphi(T_\wt)\to \CDach$ maps the extended real line $\RDach=
\R\cup\{\infty\}$ to itself. Thus we can use the Schwarz reflection
principle again to extend $f$ to the lower half-plane via $f(z):=
\overline{f(\bar{z})}$. 

\smallskip
A third description to construct the map $f$ is as follows.  
Consider the triangular lattice $\Lambda:= \Z\oplus \omega\Z$, where
$\omega= \exp(\pi i/3)$. Map the equilateral triangle with vertices
$0,1, \omega$ by a Riemann map $\wp$ to the upper half-plane, such
that $\wp(0)=-1, \wp(1)=1, \wp(\omega)=\infty$. This map extends by
reflection to a holomorphic map $\wp\colon \C\to \CDach$. It is the
\defn{Weierstra\ss\ $\wp$-function} to the lattice $\Lambda$ (slightly
differently normalized than usual). Consider now the map $z\mapsto 2z$
on $\C$. It is straightforward to check that if $z,w\in \C$ are mapped
by $\wp$ to the same point, then the same is true for $2z,2w$. Thus
there is a well defined map $f\colon \CDach\to \CDach$ (which is the
same map as before) such that the diagram
\begin{equation*}
  \xymatrix{
    \C\ar[r]^{z\mapsto 2z} \ar[d]_{\wp}
    &
    \C\ar[d]^{\wp}
    \\
    \CDach \ar[r]_f & \CDach
  }
\end{equation*}
commutes. 
Finally we note that the map $f$ is given by
\begin{equation*}
  f=\frac{2(z+1)(z-3)^3}{(z-1)(z+3)^3} -1.
\end{equation*}

\section{A sufficient criterion for mating}
\label{sec:anoth-suff-crit}

In Section~\ref{sec:hyperb-rati-maps} we saw a necessary and
sufficient criterion for a hyperbolic postcritically finite rational
map to arise as a mating. 

Here we present another sufficient criterion for a map to arise as a
mating from \cite{inv_Peano} and \cite{exp_quotients}. 

\smallskip
Recall that for an equator we required the existence of a Jordan curve
in $\CDach\setminus \post(f)$ that can be deformed by a isotopy rel.\
$\post(f)$ orientation-preserving to its preimage. 

The existence of an equator is the right
condition to check whether $f$ arises as a mating for \emph{hyperbolic
maps}. It is not the right condition however, for maps that are not
hyperbolic. Here we concentrate on the case that the map $f$ is as far away as
possible from being hyperbolic, i.e., on the case where all critical
points, hence all postcritical points are in the Julia set.  

\smallskip
Whereas in the hyperbolic case, we considered Jordan curves
\emph{avoiding} the postcritical set, we consider here Jordan curves
\emph{containing} the postcritical set. So let $\CC\subset \CDach$ be a
Jordan curve with $\post(f)\subset \CC$. The postcritical points
divide $\CC$ into closed Jordan arcs, which are called
\defn{$0$-edges}. 

We consider the preimage
\begin{equation*}
   \CC^1:=f^{-1}(\CC). 
\end{equation*}
This set can be naturally viewed
as a graph embedded in the sphere. Namely the set $\V^1:=f^{-1}(\post(f))$
is the set of vertices of this graph, the closure of one component of 
$f^{-1}(\CC)\setminus \V^1$ is an edge (called a \defn{$1$-edge}) of
this graph. All edges arise in this form. 
We call each point $p\in \V^1$ a
\defn{$1$-vertex}. Note that every postcritical point, as well as
every critical point, is a $1$-vertex.

 There may be 
multiple edges connecting two vertices, but there can be no loops.
For each $1$-edge $E^1\subset f^{-1}(\CC)$ there is a $0$-edge
$E^0\subset \CC$ such that the map $f\colon E^1 \to E^0$ is a homeomorphism.

Note that each critical point $c$ is incident to $2\deg_f (c)$
$1$-edges. In particular $\CC^1$ is not a Jordan curve, thus cannot be
isotopic to $\CC$.

Roughly speaking we demand that there is a \defn{pseudo-isotopy} rel.\
$\post(f)$ that deforms $\CC$ to $f^{-1}(\CC)$. 

\smallskip
To complete the picture, we let $X^0_\wt, X^0_\bt$ be the closures of the
two components of $\CDach\setminus \CC$. Then $X^0_\wt$ will be colored
white, and $X^0_\bt$ is colored black. They are called the two
\emph{$0$-tiles}. We orient $\CC$, so that it is positively oriented
as boundary of the white $0$-tile $X^0_\wt$. 

The closure of each component of $\CDach\setminus \CC^1$ is called a
\emph{$1$-tile}. It is not very hard to show that each $1$-tile $X$ is mapped
by $f$ homeomorphically to either $X^0_\wt$ or $X^0_\bt$, see
\cite[Chapter 4.3]{THEbook}. In the first case we color $X$ white, in
the second black. 


Recall the definition of a pseudo-isotopy from
Definition~\ref{def:pseudo-isotopy}. 
There are many ``bad'' ways in which $\CC$ be deformed by a pseudo-isotopy
to $f^{-1}(\CC)$. For example a whole interval or a Cantor set may be
deformed to a postcritical point. Also some arc $A\subset\CC$ may be
deformed to some edge in $f^{-1}(\CC)$ in a highly non-trivial
fashion. 

\begin{definition}
  \label{def:elementary_pseudo}
  An 
  \defn{elementary,
    orientation-preserving, pseudo-isotopic} deformation of $\CC$ to $\CC^1=f^{-1}(\CC)$  
  rel.\ $\post(f)$ is a pseudo-isotopy $H\colon S^2
  \times [0,1]$ rel.\ $\post(f)$ (with $H_0= \id_{S^2}$) such that
  \begin{enumerate}
  \item
    \label{item:H1}
    $H_1(\CC)=\CC^1$;
  \item 
    \label{item:H2}
    the set of points $w\in \CC$ such that $H_1(w)$ is a
    $1$-vertex is \emph{finite}. The set of all 
    such points is denoted by $\W:= (H_1)^{-1}(\V^1)\cap \CC$. Note
    that $\post(f)\subset \W$;
  \item 
    \label{item:H3}
    restricted to $\CC\setminus \W$ the homotopy $H$ is an isotopy,
    i.e.,
    \begin{equation*}
      H_1\colon \CC\setminus \W \to \CC^1\setminus \V^1 \text{ is a
        \emph{homeomorphism}.} 
    \end{equation*}
    \setcounter{mylistnum}{\value{enumi}}
  \end{enumerate}

  This means there is a bijection between $1$-edges and
  closures of components of $\CC\setminus \W$. Furthermore if $E^1\subset
  \CC^1$ is a $1$-edge and $A\subset \CC$ is the corresponding arc
  (i.e., closure of a component of $\CC\setminus \W$), then $H_1\colon A \to E^1$
  is a homeomorphism.
 
      

  \begin{enumerate}
    \setcounter{enumi}{\value{mylistnum}}
  \item 
    \label{item:H4}
    we assign an \emph{orientation} to $\CC$. For each $1$-edge
    $E^1\subset \CC^1$ there are two ways to map it homeomorphically
    to (a part of) $\CC$. 
    Namely by
    \begin{equation*}
      H_1\colon A \to E^1 \quad \text{and by}
      \quad f\colon E^1 \to E^0.
    \end{equation*}
    Here $A\subset \CC$ is a (closed) arc corresponding to $E^1$ as above, $E^0\subset
    \CC$ is a $0$-edge. 
    We demand that $H$ deforms $\CC$ \emph{orientation-preserving} to
    $\CC^1$, meaning that the orientations on $E^1$ induced by the two
    homeomorphisms above agree for each $1$-edge. 
  \end{enumerate}
  If there is a pseudo-isotopy $H$ for $\CC$ as above, we say that $f$
  has a \defn{pseudo-equator}. 
\end{definition}

\begin{remark}
  \label{rem:preH}
  Let the pseudo-isotopy $H\colon S^2 \times [0,1]\to S^2$ be as above.
  Then for any point $p\in \CC^1\setminus \V^1$ there is exactly one
  point $q\in \CC$ that is mapped by $H_1$ to $p$. For a point $c\in
  \V^1$ there are exactly $\deg_f(c)$ points in 
  $\CC$ that are mapped by $H_1$ to $c$. In this sense the curve
  $\CC$ is deformed by $H_1$ to $\CC^1$ as simply as possible.  
\end{remark}




The following is proved in \cite{inv_Peano} and \cite{exp_quotients}. 

\begin{theorem}
  \label{thm:mating_suff}
  Let $f\colon \CDach \to \CDach$ be a postcritically finite rational
  map with Julia set $\J(f)= \CDach$. Assume $f$ has a pseudo-equator
  as in Definition \ref{def:elementary_pseudo}. Then $f$ arises as
  a (topological) mating of two (postcritically finite, monic)
  polynomials $P_\wt,P_\bt$.    
\end{theorem}

\begin{remark}
  The theorem above remains true in the case when the map $f$ is not a
  rational map. Namely it holds for \emph{expanding Thurston maps}. The
  statement has to be slightly modified in the presence of periodic
  critical points. 
\end{remark}

\begin{figure}
  \centering
  \begin{overpic}
    [width=11cm, 
    tics =20]{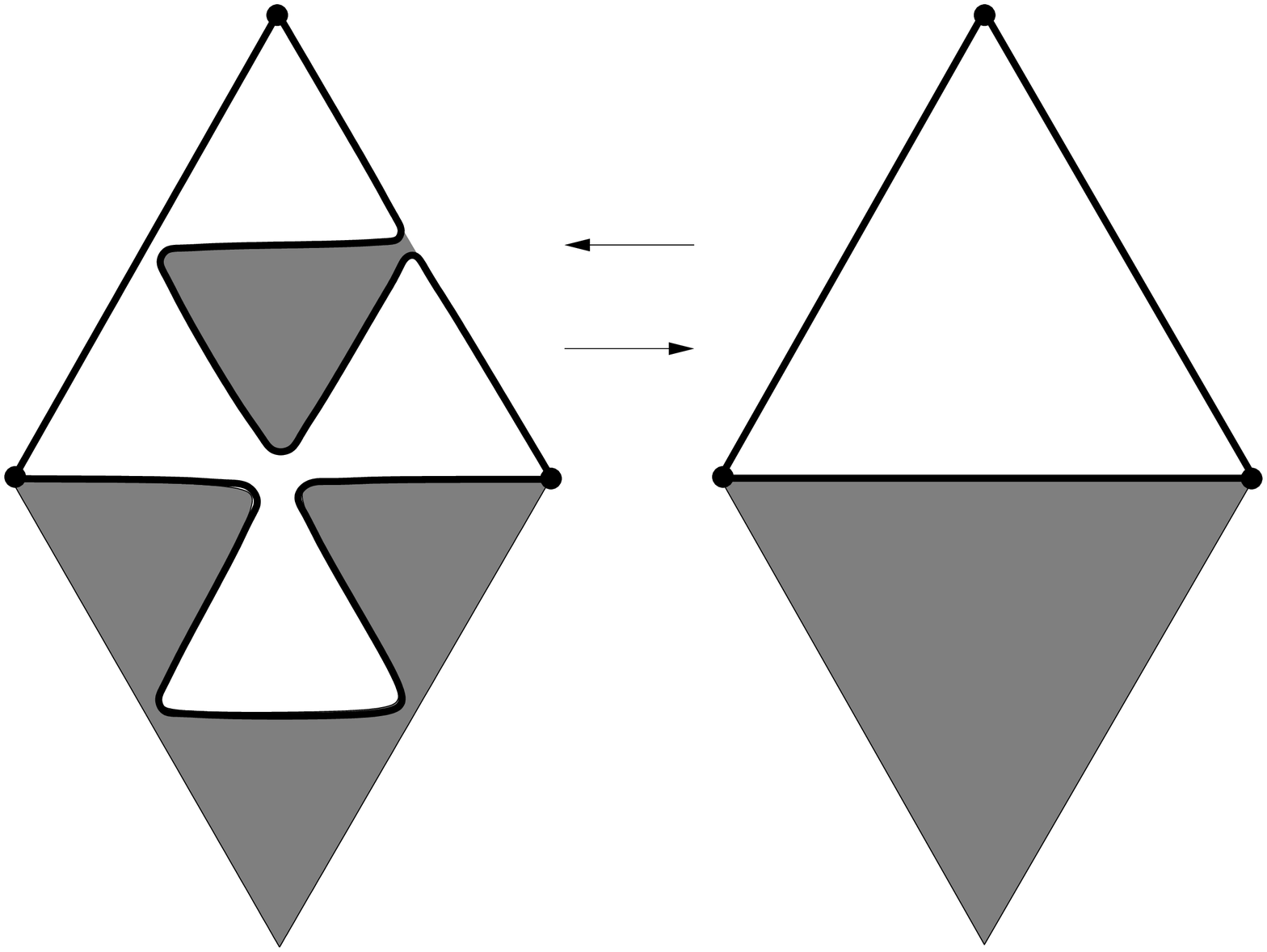}
    \put(49,50){$f$}
    \put(48,58){$H_1$}
    \put(88,61){$\CC$}
    \put(32,61){$\CC^1$}
    \put(-3,39){${\scriptstyle -1}$}
    \put(44,39){${\scriptstyle 1}$}
    \put(53.5,39){${\scriptstyle -1}$}
    \put(100,39){${\scriptstyle 1}$}
    \put(21,76){${\scriptstyle \infty}$}
    \put(77,76){${\scriptstyle \infty}$}
    \put(77,39){${\scriptstyle E_0}$}
    \put(93,50){${\scriptstyle E_1}$}
    \put(66,60){${\scriptstyle E_2}$}
    \put(21,37){${\scriptstyle c_0}$}
    \put(8.5,56){${\scriptstyle c_1}$}
    \put(8.5,18){${\scriptstyle c_1}$}
    \put(33,56){${\scriptstyle c_2}$}
    \put(33,18){${\scriptstyle c_2}$}
  \end{overpic}  
  \caption{Deforming $\CC$ to $\CC^1$.}
  \label{fig:2}
\end{figure}

\begin{example}
  We consider the example from Section~\ref{sec:an-example}. Recall
  that $\post(f)=\{-1,1,\infty\}$. These points are the vertices of
  the pillow in Figure~\ref{fig:1}. 
  We choose $\CC= \widehat{\R}$. In
  the model of the map indicated in Figure~\ref{fig:1}, $\CC$ is the
  common boundary of the two triangles which forms the pillow. The
  preimage $\CC^1:=f^1(\CC)$ is the union of all the edges of the
  small triangles to the left in Figure~\ref{fig:1}. 

  The pseudo-isotopy $H$ that deforms $\CC$ elementary,
  orientation-preserving, to $\CC^1=f^{-1}(\CC)$ is indicated (to the
  left) in Figure~\ref{fig:2}.

\end{example}

\begin{figure}
  \centering
  \begin{overpic}
    [width=12cm, 
    tics =20]{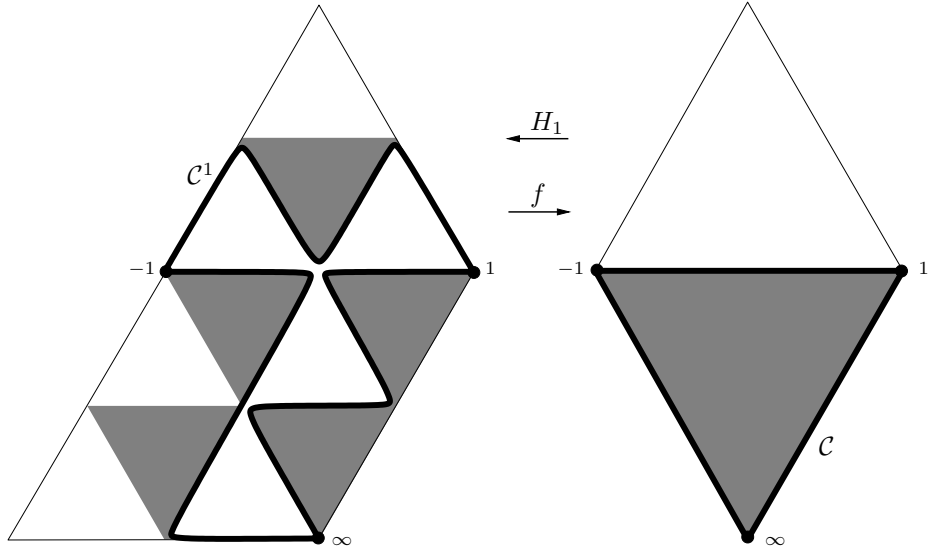}
    \put(58,46){$H_1$}
    \put(58,38){$f$}
    \put(90,10){$\CC$}
    \put(20,40){$\CC^1$}
    \put(101,30){${\scriptstyle 1}$} 
    \put(61,30){${\scriptstyle -1}$} 
    \put(84,0){${\scriptstyle \infty}$} 
    \put(13.5,30){${\scriptstyle -1}$}
    \put(53,30){${\scriptstyle 1}$}
    \put(36,0){${\scriptstyle \infty}$}   
  \end{overpic}  
  \caption{Orientation reversing pseudo-isotopy.}
  \label{fig:3}
\end{figure}

\begin{example}
  \label{ex:orient_reverse}
  We consider again the same example from
  Section~\ref{sec:an-example}. Again $\CC=\widehat{\R}$. We have
  chosen to draw the curve $\CC$ on the right of Figure~\ref{fig:3} as
  the boundary of the black $0$-tile, however. In
  Figure~\ref{fig:3} we show a pseudo-isotopy that deforms $\CC$ to
  $\CC^1$ in an \emph{orientation-reversing} way. By this we mean that
  the shown 
  pseudo-isotopy $H$ satisfies all properties from
  Definition~\ref{def:elementary_pseudo} except \eqref{item:H4}. This
  is seen as follows. If we traverse $\CC$ positively as boundary of
  the white $0$-tile, we go through the postcritical points in the
  cyclic order $-1 \to 1\to \infty\to -1$. However if we traverse
  $\CC^1:= f^{-1}(\CC)$ in the orientation given by $f$ (i.e., each
  $1$-edge is traversed positively as boundary of the white $1$-tile
  in which it is contained), then the postcritical points are
  traversed in the order $-1,\infty, 1$. More precisely, we see that
  the orientations on each $1$-edge $E^1$ induced by $f$ and by $H_1$
  are opposite.
\end{example}

 If we compare Figure~\ref{fig:2} with
Figure~\ref{fig:3} we see one difference. Namely in the pseudo-isotopy
indicated in Figure~\ref{fig:2} the interior of $\CC$, more precisely
the interior of the white $0$-tile $X^0_\wt$, is deformed to all white
$1$-tiles. On the other hand the pseudo-isotopy indicated in
Figure~\ref{fig:3} deforms the interior of the black $0$-tile to the
white $1$-tiles. This is a general phenomenon.

\begin{lemma}
  \label{lem:Hdeformswb}
  Let $H$ be an elementary, pseudo-isotopic deformation of $\CC$ to
  $\CC^1$, i.e., a pseudo-isotopy rel.\ $\post(f)$ that satisfies
  \eqref{item:H1}, \eqref{item:H2}, and \eqref{item:H3} (but not
  necessarily \eqref{item:H4}) of Definition
  \ref{def:elementary_pseudo}.  Then

  \begin{enumerate}
  \item 
    \label{item:Horient1}
    Either $H$ deforms $\CC$ to $\CC^1$ orientation-preserving (i.e.,
    satisfies \eqref{item:H4} of Definition~\ref{def:elementary_pseudo}),
    or $H$ deforms $\CC$ to $\CC^1$ \emph{orientation-reserving}. The
    latter means that for any $1$-edge the orientations induced by
    the homeomorphisms 
    $$f\colon E^1\to E^0 \text{ and }H_1\colon A\to E^1$$
    disagree.  Here $E^0\subset \CC$ is a $0$-edge and $A\subset \CC$ is an arc
    as in Definition~\ref{def:elementary_pseudo}~\eqref{item:H3}. 

    %
    \setcounter{mylistnum}{\value{enumi}}
  \end{enumerate}

  Let $x,y\in (H_1)^{-1}(\CC^1)$, i.e., two points that are
  mapped by $H_1$ to the interior of (possibly distinct)
  $1$-tiles. Let $X^0\ni x$, $Y^0\ni y$ be the $0$-tiles containing
  $x,y$, and $X^1\ni H_1(x)$, $Y^1\ni H_1(y)$ be the $1$-tiles
  containing $H_1(x), H_1(y)$. Then

  \begin{enumerate}
    \setcounter{enumi}{\value{mylistnum}}
    
    \item
      \label{item:Horient2}
      \begin{align*}
        X^0, Y^0 \text{ have the same color }
        \\
        \Longleftrightarrow X^1, Y^1 \text{ have the same color.}
      \end{align*}
      
      
    \item 
      \label{item:Horient3}
      \begin{align*}
        & X^0 \text{ and } X^1 \text{ have the same color}
        \\
        \iff
        &H \text{ is orientation-preserving,}    
      \end{align*}
      i.e., $H$ satisfies \eqref{item:H4} of
      Definition~\ref{def:elementary_pseudo}. 
  \end{enumerate}
\end{lemma}

\begin{proof}
  \eqref{item:Horient1}
  Consider a critical point, i.e., a $1$-vertex $c$. This point is
  contained in the $1$-edges $E_0, \dots, E_{n-1}$ (where
  $n=2\deg_f(c)$), which are  
  labeled mathematically positively around $c$. 

  \smallskip
  We first note that the
  orientation on these $1$-edges induced by $f$ alternates. This is
  seen as follows. There are $n$ black as well as $n$ white $1$-tiles
  around $c$, the colors of the $1$-tiles around $c$ alternate. The
  endpoint of the edge $E_j$ is $c$ (via the orientation given by $f$)
  if and only if the sector between $E_{j-1},E_j$ contains a white
  $1$-tile (and the sector between $E_j, E_{j+1}$ contains a black
  $1$-tile). Conversely the initial point of $E_j$ is $c$ if and only
  if the sector between $E_{j-1}, E_j$ contains a black $1$-tile (and
  the sector between $E_j, E_{j+1}$ contains a white $1$-tile).

  \smallskip
  We now consider the orientation on $1$-edges induced by the
  pseudo-isotopy $H$.

  Let $A, A'\subset \CC$ be two adjacent arcs that are deformed
  to $1$-edges $E,E'$ as in
  Definition~\ref{def:elementary_pseudo}~\eqref{item:H3}, where $A\cap
  A'$ is deformed to $c$ by $H$. We assume
  that $A'$ succeeds $A$ with respect to the orientation of
  $\CC$. Then (with respect to the orientation induced by $H$) the
  endpoint of $E$ is $c$, while $c$ is the initial point of $E'$. 

  Since $H$ is a pseudo-isotopy $H_{1-\epsilon}$ is a homeomorphism
  for all $0<\epsilon<1$. It follows that in the sector between $E,
  E'$ there is an even number of $1$-edges. 

  Thus from the first claim it follows that the orientation on $E$
  induced by $H$ agrees with the orientation on $E$ induced by $f$ if
  and only if the orientations on $E'$ induced by $H$ and $f$ agree. 

  \smallskip
  Consider the arc $A''$ on $\CC$ that succeed $A'$. This is deformed
  by $H$ to a $1$-edge $E''$. The orientations induced on $E''$ by $f$
  and $H$ agree if and only if the respective orientations on $E'$
  agree. Continuing in this fashion we obtain the statement. 







\smallskip
\eqref{item:Horient2}
Let $x^1:= H_1(x), y^1:=H_1(y)$. 
We want to consider the winding number. To be able to do that we 
assume that $\infty$ is contained in the interior of the black
$0$-tile, furthermore we assume that $H$ deforms $\infty$ neither to
$\CC^1$ nor to $x^1$ or $y^1$. Then we can define the winding
number $N_\CC(x), N_\CC(y)$ of $\CC$ for $x$ and $y$. 

Then $N_\CC(x), N_\CC(y)$ agree if and only if $X^0, Y^0$ have the
same color. This happens if and only if $N_{\CC^1}(x^1),
N_{\CC^1}(y^1)$ agree, these winding numbers are computed by mapping   
$\CDach\setminus \{H_1(\infty)\}$ (orientation preserving) to
$\C$, i.e., identifying $\infty^1:= H_1(\infty)$ with $\infty$. The
curve $\CC^1= H_1(\CC)$ is traversed in the direction 
induced by $H_1$ as well as the orientation of $\CC$. 

Assume that $H_1(\infty)$ is contained in a black $1$-tile. 
Each $1$-edge is contained in exactly one white $1$-tile. Consider the 
$1$-edges $E_1,\dots, E_{k}$ contained in the boundary of a white
$1$-tile $Z$. By 
\eqref{item:Horient1} all these $1$-edges are either positively or
negatively oriented as boundary of $Z$. In the first case the winding 
number of $E_1\cup \dots \cup E_{k}$ is $1$, in the second case $-1$,
for all points in the interior of $Z$. For all points in the
complement of $Z$ the winding number is $0$. The same argument applies
to all white $1$-tiles. Note that by \eqref{item:Horient1} the
boundaries of all white $1$-tiles have the same orientation induced by
$H_1$, i.e., are either all positively or all negatively oriented. 
Since taking the union of all $1$-edges in the 
boundaries of all white $1$-tiles yields all $1$-edges, it follows that
the winding number of $\CC^1$ 
for all points in the interior of some white $1$-tile is either $1$
or $-1$, while the winding number for all points in the interior of
black $1$-tiles is $0$. This finishes the claim. The argument in the
case when $H_1(\infty)$ is contained in a white $1$-tiles is
completely analogous. 

\smallskip
\eqref{item:Horient3}
We use the setting as above. Assume $X^0$ is white, then the winding
number of $\CC^1$ for $x^1$ is $1$ if and only if the orientation
induced by $H_1$ of the $1$-edges in the boundary of $X^1$ agrees with
the orientation of them as boundary of $X^1$, this happens if and only
if $H$ is orientation-preserving by \eqref{item:Horient1}. If $X^0$ is
white the argument is completely analogous. 

\end{proof}

\section{Connections}
\label{sec:connections}

There is an equivalent way to describe the existence of a
pseudo-equator as in
Definition~\ref{def:elementary_pseudo}. Intuitively the  
description is most easily explained along a picture as in
Figure~\ref{fig:2}.  

Namely consider the left picture in Figure~\ref{fig:2}. Each of
the critical points $c_j$ is contained in exactly $3$ white, as well
as $3$ black $1$-tiles. Consider first the critical point $c_0$. Here
all white $1$-tiles are \emph{connected} at $c_0$, while none of the
black $1$-tiles are connected at $c_0$. 

At the critical point $c_1$ there are two white $1$-tiles connected,
and two black $1$-tiles that are connected. Also there is one white,
as well as one black $1$-tile that is not connected to any other
$1$-tile at $c_1$.  

At the critical point $c_2$ all three black $1$-tiles containing $c_2$
are connected, all three white $1$-tiles containing $c_2$ are not
connected to any other $1$-tile. 

The connection of white $1$-tiles at each critical point $c_j$ is
\emph{complementary} to the connection of black $1$-tiles at
$c_j$. The white, as well as the black $1$-tiles, are connected in such
a way that the resulting \emph{white connection graph} is a \emph{spanning
  tree}. 

The connection of white $1$-tiles may be \emph{represented
  geometrically} as in Figure~\ref{fig:2}. Taking the boundary of this
geometric representation of the white cluster results in a curve that
is isotopic to $\CC$ rel.\ $\post(f)$. 

\smallskip
We now proceed to make the above precise.




\smallskip
Let $X_0,\dots X_{2n-1}$ be the $1$-tiles intersecting in a
$1$-vertex $v$, ordered mathematically positively around $v$. The
white $1$-tiles have even index, the black ones odd index. We consider
a decomposition $\pi_\wt=\pi_\wt(v)$ of $\{0, 2, \dots , 2n-2\}$ (i.e., of
indices corresponding to white $1$-tiles around $v$); and a
decomposition $\pi_\bt=\pi_\bt(v)$ of $\{1,3,\dots, 2n-1\}$ (i.e., of
indices corresponding to black $1$-tiles around $v$). They satisfy the
following:
\begin{itemize}
\item They are \emph{decompositions}. This means $\pi_\wt=\{b_1,\dots,
  b_N\}$, where each \defn{block} $b_i$ is a subset of $\{0,2,\dots,
  2n-2\}$, $b_i\cap b_j=\emptyset$ ($i\neq j$), and $\bigcup b_i =
  \{0,2, \dots, 2n-2\}$. Similarly for $\pi_\bt$. 
\item The decompositions $\pi_\wt,\pi_\bt$ are \defn{non-crossing}. This
  means the following. Two distinct blocks $b_i,b_j\in \pi_\wt$ are
  \emph{crossing} if there are numbers $a,c\in b_i$, $b,d\in b_j$ and
  \begin{equation*}
    a < b < c < d. 
  \end{equation*}
  Each partition $\pi_\wt,\pi_\bt$ does not contain any (pair of) crossing
  blocks. 
\item The partitions $\pi_\wt,\pi_\bt$ are \defn{complementary}. This
  means the following. Given $\pi_\wt$, the partition $\pi_\bt$ is the
  unique, biggest partition (of $\{1,3,\dots, 2n-1\}$) such that
  $\pi_\wt\cup \pi_\bt$ is a non-crossing partition of $\{0,1,\dots,
  2n-1\}$.  
\end{itemize}
A partition $\pi_\wt\cup \pi_\bt$ as above is called a \defn{complementary
  non-crossing partition}, or cnc-partition. 
We may \emph{represent} a cnc-partition \emph{geometrically} as
follows. Let $i,j\in b\in \pi_\wt$. We call these indices
\emph{succeeding} (in $b$), if $i+1, i+2, \dots, j-1 \;(\bmod\, 2n)
\notin b$. Let $e_k:= \exp(2\pi k i/2n)$, $k=0, \dots, 2n-1$ be the $2n$
unit roots. Consider the closed unit disk $\Dbar$. For each pair of
succeeding indices $i,j$ in any white block $b\in \pi_\wt$  we draw a
Jordan arc $g_m$ ($m=1,\dots, n$) in $\Dbar$ connecting $e_{i+1},
e_j$. Each such arc intersects $\partial \D$ only in its endpoints
(i.e., in $e_{i+1}, e_j$). Furthermore two distinct such arcs are
disjoint. Figure~\ref{fig:connection} shows a geometric representation 
of the cnc-partition $\pi_\wt\cup \pi_\bt$ given by
$\pi_\wt=\left\{\{0,4,6\}, \{2\}\right\}$,
$\pi_\bt=\left\{\{1,3 \},\{5\},\{7\}\right\}$. 
 
\begin{figure}
  \centering
  
  \includegraphics[width=11cm]{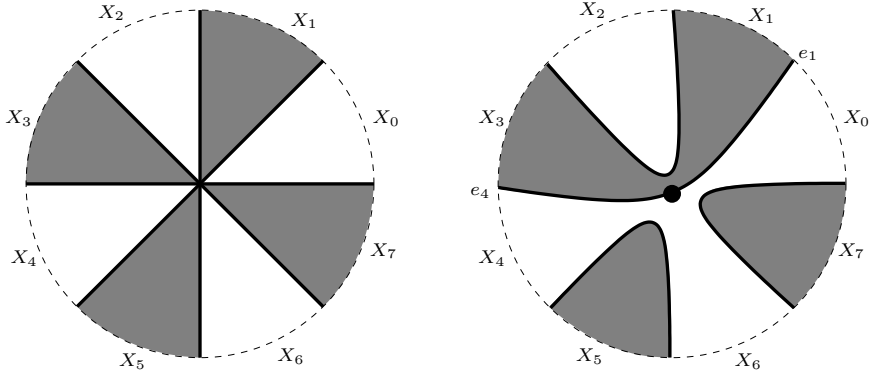}
  \begin{picture}(10,10)
    %
    \put(-220,0){$\scriptstyle{X_6}$}
    \put(-185,40){$\scriptstyle{X_7}$}
    \put(-184,90){$\scriptstyle{X_0}$}
    \put(-215,127){$\scriptstyle{X_1}$}
    \put(-288,130){$\scriptstyle{X_2}$}
    \put(-323,90){$\scriptstyle{X_3}$}
    \put(-321,37){$\scriptstyle{X_4}$}
    \put(-280,-2){$\scriptstyle{X_5}$}
    %
    \put(-47,-2){$\scriptstyle{X_6}$}
    \put(-8,40){$\scriptstyle{X_7}$}
    \put(-6,90){$\scriptstyle{X_0}$}
    \put(-42,129){$\scriptstyle{X_1}$}
    \put(-106,131){$\scriptstyle{X_2}$}
    \put(-144,90){$\scriptstyle{X_3}$}
    \put(-144,37){$\scriptstyle{X_4}$}
    \put(-107,0){$\scriptstyle{X_5}$}    
    \put(-23,115){$\scriptstyle{e_1}$}
    \put(-147,63){$\scriptstyle{e_4}$}        
  \end{picture}
  \caption{Connection at a vertex.}
  \label{fig:connection}
\end{figure}

The arcs $g_m$ divide $\Dbar$ into $n+1$ components. We color those
components in a checkerboard fashion black and white, such that
components which share an arc $g_m$ as its boundary have different
color. Furthermore we color the component having the circular arc
between $e_0,e_1$ on $\partial \D$ white. 

A geometric representation of a cnc-partition is to be thought of as a   
blow-up of the $1$-edges incident to a $1$-vertex $v$. We will
sometimes require to know where in the geometric representation the
original vertex is located. This is achieved by \emph{marking} the
cnc-partition. This means we mark one of the arcs $g_m$
above. Equivalently we may mark a pair of succeeding indices $i,j$
contained in some white block $b\in\pi_\wt$. If a cnc-partition is
marked we always let the marked arc $g_m$ contain the origin. 

In Figure~\ref{fig:2} the arc connecting $e_{0+1},e_4$, equivalently
the succeeding indices $0,4$ of the block $\{0,4,6\}$, is marked.

\medskip
A \defn{connection} (of
$1$-tiles) assigns to each $1$-vertex $v$ a cnc-partition $\pi_\wt(v)
\cup \pi_\bt(v)$ as above. Furthermore in the case when $v=p$ is a
postcritical point, the cnc-partition is marked. This marking is
mostly relevant in the case when the postcritical point is at the same
time a critical point, since then there are several choices which arc
to mark. 
Two
$1$-tiles $X_i,X_j\ni v$ are said to be \defn{connected at $v$} if
the indices $i,j$ are contained in the \emph{same} block of $\pi_\wt(v)\cup
\pi_\bt(v)$. Note that tiles of different color are never connected. The
$1$-tile $X_i$ is \defn{incident} (at $v$) to the block $b\ni i$ of $\pi_\wt(v)\cup
\pi_\bt(v)$.  

The \emph{white connection graph} is defined as follows. For each
white $1$-tile there is a vertex and for each block $b\in \pi_\wt(v)$
(for any $1$-vertex $v$) there is a vertex. There is an edge in the white
connections graph between $b\in \pi_\wt(v)$ and the white $1$-tile
$X\ni v$
if and only if $X$ is incident to $b$ at $v$. 

A \emph{geometric representation} of a connection is achieved as
follows. For each $1$-vertex $v$ a neighborhood of $v$ ``looks like''
the picture on the left in Figure~\ref{fig:connection}. This is
replaced by the picture on the right in Figure~\ref{fig:connection},
i.e., by a geometric representation of the cnc-partition that
represents the connection at $v$. More precisely there is a
neighborhood $U_v$ of $v$ and a homeomorphism $\varphi=\varphi_v
\colon U_v\to 
\D$ with the following properties. The point $v$ is mapped to the
origin. For any $1$-edge $E\ni 
v$, the map $\varphi$ maps $E\cap U_v$ to some ray $R_k=\{r\exp(2\pi k
i/2n)\mid 0\leq r<1\}$, where $k=0,\dots, 2n-1$. If $X\ni v$ is a
white $1$-tile, then $X\cap U_v$ is mapped by $\varphi$ to the sector
between two rays $R_{2k}, R_{2k+1}$. Finally distinct $1$-vertices
$v,w$ have disjoint neighborhoods $U_v, U_w$. Draw a geometric
representation of the cnc-partition $\pi_\wt(v)\cup \pi_\bt(v)$ in
$\D$ as above. Replace $U_v$ by the preimage (of this geometric
representation) by $\varphi_v$ (for each $1$-vertex $v$).  

\smallskip
Given a geometric representation of a connection we obtain several
black and white components. The following is \cite[Lemma
6.23]{inv_Peano}. 

\begin{lemma}
  \label{lem:conn_graph_geom}
  Assume a connection (of $1$-tiles) is given. Then the following are
  equivalent.
  \begin{itemize}
  \item The white connection graph is a spanning tree.
  \item In each geometric representation (of the connection) there is
    a single white component $V_\wt$, which is a Jordan domain.
  \end{itemize}
\end{lemma}

The following is \cite[Lemma~7.2]{inv_Peano}. 

\begin{lemma}
  \label{lem:equ_conn_pseudo_eq}
  Let $f\colon \CDach \to \CDach$ be a postcritically finite rational
  map with Julia set $\J(f)= \CDach$ and $\CC\subset \CDach$ be a
  Jordan curve with $\post(f)\subset \CC$. Then
  \begin{itemize}
  \item there is a pseudo-isotopy $H\colon \CDach \times [0,1] \to
    \CDach$ as in Definition~\ref{def:elementary_pseudo} (i.e., $f$
    has a pseudo-equator) if and only if 
  \item there is a connection of $1$-tiles, such that the white
    connection graph is a spanning tree and the boundary $\partial
    V_\wt$ (of the white component $V_\wt$ from
    Lemma~\ref{lem:conn_graph_geom}) is orientation-preserving
    isotopic to $\CC$ rel.\ $\post(f)$.   
  \end{itemize}
\end{lemma}

Here $\partial V_\wt$ is positively oriented as boundary of the white
component $V_\wt$ (and $\CC$ is positively oriented as boundary of the
white $0$-tile $X^0_\wt$ as before). 

\begin{theorem}
  \label{thm:P3mate}
  Let $f\colon \CDach \to \CDach$ be a postcritically finite rational
  map, with Julia set $\J(f)=\CDach$ and $\#\post(f)=3$. Then $f$ or
  $f^2$ arises as a mating. 
\end{theorem}
 
\begin{proof}
  Let $f\colon \CDach \to \CDach$ be as in the statement. 
  Let $\CC\subset \CDach$ be a Jordan curve with $\post(f)\subset
  \CC$. We choose the closure of one component of $\CDach\setminus
  \CC$ to be $X^0_\wt$. Then $\CC$ is positively oriented as boundary
  of the white $0$-tile 
  $X^0_\wt$. It is easy to construct a connection of white $1$-tiles
  such that the resulting white connection graph is a spanning tree,
  see \cite[Corollary 6.20]{inv_Peano}. Indeed one starts with a
  connection where no white $1$-tiles are connected at any $1$-vertex,
  successively one ``adds'' other $1$-tiles, until one obtains a
  spanning tree as desired. 

  From now on the connection (of white $1$-tiles), which results in
  the white connection graph being a spanning tree, is fixed.  
  Let $V_\wt$ be the white component, according to
  Lemma~\ref{lem:conn_graph_geom}. Let $\partial V_\wt$ be its
  boundary, oriented positively as boundary of $V_\wt$. 

  \smallskip
  The following
  fact is well-known (a proof can however be 
  found in \cite[Chapter 9.2]{THEbook}). Let $P\subset S^2$ be a set
  with $\# P \leq 3$ and $\gamma,\gamma'\subset S^2$ be two Jordan
  curves, which both contain $P$. Then $\gamma, \gamma'$ are isotopic
  rel.\ $P$. However if $\gamma,\gamma'\subset S^2$ are oriented it is
  not necessarily true that $\gamma,\gamma'$ are
  \emph{orientation-preserving} isotopic rel.\ $P$.  
  
  \smallskip
  If $\partial V_\wt$ is orientation-preserving isotopic to $\CC$
  rel.\ $\post(f)$ we are done by Lemma~\ref{lem:equ_conn_pseudo_eq}
  and Theorem~\ref{thm:mating_suff}, i.e., $f$ arises as a mating. 

  \smallskip
  Assume $\partial V_\wt$ is not orientation-preserving isotopic to
  $\CC$ rel.\ $\post(f)$. Then $\partial V_\wt$ is
  orientation-reversing isotopic to $\CC$ rel.\ $\post(f)$. We obtain
  an elementary, orientation-reversing pseudo-isotopy $H$, that
  deforms $\CC$ to $\CC^1$ as in
  Lemma~\ref{lem:Hdeformswb}~\eqref{item:Horient1}.   

  \smallskip
  Let $\widetilde{H}\colon \CDach \times [0,1] \to \CDach$ be the lift
  of $H$ by $f$ with $\widetilde{H}_0=\id_{\CDach}$. We note some properties of
  the lift $\widetilde{H}$, proofs may be found in \cite[Lemma~3.5 and
  Lemma~3.6]{inv_Peano}.
  Namely $\widetilde{H}$ is a pseudo-isotopy rel.\
  $\V^1$ ($=$ set of 
  $1$-vertices) such that $f\circ \widetilde{H}(x,t) =  H(f(x), t)$
  for all $x\in \CDach$, $t\in [0,1]$. It deforms $\CC^1$ to $\CC^2:=
  f^{-1}(\CC^1) = f^{-2}(\CC)$, i.e., $\widetilde{H}_1 (\CC^1)=
  \CC^2$. 
  We call $\V^2:=f^{-2}(\post(f))$ the set of $2$-vertices, the
  closure of one component of $\CC^2\setminus \V^2$ is called a
  $2$-edge. Only finitely many points $x\in \CC^1$ are deformed to any 
  $2$-vertex (i.e., to any point $v\in f^{-2}(\post(f))$). The points
  $x$ as before divide $\CC^1$ into closed arcs. There is a bijection
  between such closed arcs $\widetilde{A}\subset \CC^1$ and
  $2$-edges. Finally for each such arc $\widetilde{A}\subset \CC^1$
  there is a $1$-edge $E$ and the map $\widetilde{H}_1\colon
  \widetilde{A} \to E$ is a homeomorphism.

  \begin{claim}
    $\widetilde{H}$ deforms $\CC^1$ to $\CC^2$ in an
    orientation-reversing way. 

    \mbox{}    
    As before this means that the orientations induced on $\CC^2$ by
    $f\colon \CC^1 \to \CC^2$ and $\widetilde{H}_1\colon \CC^1\to \CC^2$
    disagree.  

    Each $2$-edge $\widetilde{E}$ is mapped
    by $f$ homeomorphically to a $1$-edge $E\subset \CC^1$ (see
    \cite[Chapter~5.3]{THEbook}). Recall from
    Definition~\ref{def:elementary_pseudo}~\eqref{item:H2} that there
    is a closed arc $A\subset \CC$ that is deformed by $H$ to $E$. 
    Since $\widetilde{H}$ is the lift of $H$ by $f$, there is an arc
    $\widetilde{A}\subset \CC^1$ that is mapped homeomorphically to
    $\widetilde{E}$ by $\widetilde{H}_1$, and $f(\widetilde{A})=
    A$. Note that $f\colon \widetilde{A}\to A$, as well as $f\colon \colon
    \widetilde{E} \to E$, is orientation-preserving. Thus it follows
    that $H$ is orientation-reversing if and only if $\widetilde{H}$
    is orientation-reversing, proving the claim. The argument is
    worked out in more detail in \cite[Lemma~3.12]{inv_Peano}. 
  \end{claim}

  Consider $K\colon \CDach \times [0,1] \to \CDach$ given by $K(x,t):=
  \widetilde{H}(H(x,t),t)$. This is a pseudo-isotopy rel.\ $\post(f)$
  that deforms $\CC$ elementary, orientation-preserving to $\CC^2=
  f^{-2}(\CC)$. Thus it follows from Theorem~\ref{thm:mating_suff}
  that $f^2$ arises as a mating. 
\end{proof}

\begin{remark}
  We do not know whether it is necessary to take the second iterate
  $f^2$ in the previous theorem. 
\end{remark}

\section{Critical portraits}
\label{sec:critical-portraits}

In the next section we will describe an algorithm to unmate a rational 
map $f$. More precisely, from a pseudo-equator as in
Definition~\ref{def:elementary_pseudo} we can recover the white and
black polynomial $P_\wt, P_\bt$ that yield $f$ as their mating. We
need however a description of polynomials that is adapted to the
situation. Namely we need a description in terms of external angles. 

Fix an integer $d\geq 2$ ($d$ will be the degree of the rational map
$f$ as well as the polynomials $P_\wt, P_\bt$). The map $\mu\colon
\R/\Z \to \R/\Z$ is $\mu(t):= dt \bmod 1$. We will somehow abuse
notation by identifying a point $x\in \R$ with the corresponding
equivalence class $[x]\in \R/\Z$, similarly we identify $q\in \Q$ with
the corresponding $[q]\in \Q$.  
\begin{definition}
\label{def:crit_portrait}
  A list $\mathcal{A}=A_1,\dots, A_m$
  is called a \defn{critical portrait} if conditions
  (CP~\ref{item:CP1})--(CP \ref{item:CP6}) are satisfied.
  \begin{enumerate}[\upshape(CP 1)]
  \item 
    \label{item:CP1}
    Each $A_j\subset \Q/\Z\subset \R/\Z$ is a finite set, distinct
    sets $A_i,A_j$ are disjoint;  
  \item $\mu$ maps each set $A_j$ to a single point,
    \begin{equation*}
      \mu(A_j)= \{a_j\},
    \end{equation*}
    for all $j=1,\dots, m$;
  \item $\sum_j \left(\#A_j -1\right) = d - 1$.
  \item The sets are \emph{non-crossing}. This means the following. Two
    distinct sets $A_i,A_j$ are called \emph{crossing}, if there are
    (representatives) $s,u\in
    A_i, t,v\in A_j$ such that $0\leq s< t< u< v \leq 1$, otherwise
    \emph{non-crossing}. All distinct sets $A_i,A_j$ of the critical
    portrait are non-crossing.
  \item No $a\in A_j$ is periodic under $\mu$ (for any $j=0,\dots,
    m-1$). 
    \setcounter{mylistnum}{\value{enumi}}
  \end{enumerate}
  The set $\A:= \bigcup \{\mu^k(A_j) \mid j=1,\dots, m, \; k\geq 1\}$,
  i.e., the union of forward orbit of all angles in any of the sets $A_j$, then
  is a finite set. 
  \begin{enumerate}[\upshape(CP 1)]
    \setcounter{enumi}{\value{mylistnum}}
  \item No set $A_j$ contains more than one point of $\A$. 
    \setcounter{mylistnum}{\value{enumi}}
  \end{enumerate}
  The following definition is somewhat technical. The reader should
  think of $\R/\Z=S^1$ as being the boundary of the unit disk $\D$. We
  form the convex hull of each set $A_j$ (with respect to the
  hyperbolic metric on $\D$. We remove all these hulls. The closure of
  one remaining component is called a \emph{$1$-gap}. Points in the same
  gap should be thought of as  being not separated by the sets
  $A_j$. We are really only interested in points on the unit circle. 

  Here is the formal definition. The points $\bigcup A_j$ divide the
  circle into closed intervals $[a,b]$, i.e., $a,b\in \bigcup A_j$ and $(a,b)$
  does not contain any point from $\bigcup A_j$. A $1$-gap is a
  union of such intervals. 
  Two intervals
  $[a_1,b_1], [a_2, b_2]$ belong to the same \emph{$1$-gap} if and
  only if for any two points $c_1\in (a_1,b_1), c_2 \in (a_2,b_2)$,
  the sets $\{c_1, c_2\},  A_j$ are non-crossing for all
  $j=1,\dots,n$. 

  Two points $x,y\in \R/\Z$ belong to the same \emph{$n$-gap} ($n\geq
  1$) if
  $\mu^k(x),\mu^k(y)$ belong to the same $1$-gap for all $k=0, \dots ,
  n-1$. 
  \begin{enumerate}[\upshape(CP 1)]
    \setcounter{enumi}{\value{mylistnum}}
  \item 
    \label{item:CP6}
    There is a constant $n_0\in \N$ such that the following
    holds. Two distinct points $a,b\in \A$ are not contained in the
    same $n_0$-gap. 
  \end{enumerate}
\end{definition}

\begin{theorem}[Bielefeld-Fisher-Hubbard \cite{MR1149891}]
  \label{thm:poirier}
  Let $\mathcal{A}$ be a critical portrait as in
  Definition~\ref{def:crit_portrait}. Then there is a (unique up to
  affine conjugacy) monic polynomial $P$ of degree $d$ realizing
  it. The polynomial $P$ is postcritically finite, each critical
  point of $P$ is strictly preperiodic.  
\end{theorem}

That a polynomial $P$ realizes a critical portrait means that for each
$A_j$ there is a critical point $c_j$ of $P$. The degree of $P$ at
$c_j$ is $\#A_j$. If $a\in A_j$ then the external ray $R(a)$ lands at
$c_j$. There is a generalization of the above theorem due to
Poirier\cite{MR2496235}.  

\section{Unmating the map}
\label{sec:find-polyn-mating}

A pseudo-equator for $f$ as in Definition~\ref{def:elementary_pseudo}
does not only guarantee that $f$ arises as a mating, but it is
possible to explicitly find the polynomials that when mating give the
map $f$. This is described here. 

\smallskip
Let $\CC\supset \post(f)$ be a pseudo-equator for the (rational,
postcritically finite) map $f\colon \CDach \to \CDach$, whose Julia
set is the whole sphere. 

Let $X^0_{w}, X^0_{b}$ be the two $0$-tiles (defined in terms of
$\CC$) which are colored white and black. We orient $\CC$, so that it
is positively oriented as boundary of (the white $0$-tile) $X^0_\wt$. 

Recall that the postcritical points divide $\CC$ into (closed)
$0$-edges $E_0, \dots, E_{k-1}$, which we label positively on $\CC$. 

Consider now a $1$-edge $E^1$. We say it is of \emph{type} $j$ if
$f(E^1)= E_j$. 
Each $0$-edge is deformed by $H_1$ into several $1$-edges. We record
how many $1$-edges of each type are contained in such a deformed
$0$-edge in the matrix $M=(m_{ij})$ defined as follows:
\begin{equation*}
  m_{ij} := \text{ number of $1$-edges of type $j$ contained in
    $H_1(E_i)$.}  
\end{equation*}
This matrix is called the \emph{edge replacement matrix} of the
pseudo-isotopy $H$. 
Since there are exactly $n$ $1$-edges of each type it follows that
$\sum_i m_{ij}=d=\deg f$. 
It is relatively easy to show that the matrix $M$ is \emph{primitive},
i.e., that $M^n> 0$ for some $n\in \N$. Thus it follows from the
Perron-Frobenius theorem that $d$ is a simple eigenvalue (which is in
fact the spectral radius of $M$), with eigenvector
$l=(l_j)>0$. We normalize $l$ by $\sum l_j=1$, this makes $l$
unique. 

To illustrate we consider the pseudo-isotopy shown in
Figure~\ref{fig:2}. Here the edge replacement matrix and the
corresponding eigenvector is
\begin{equation*}
  M= 
  \left(
  \begin{array}{ccc}
    2 & 2 & 1
    \\
    2 & 1 & 2
    \\
    0 & 1 & 1
  \end{array}
  \right),
  \quad
  l=\frac{1}{15}
  \left(
    \begin{array}{c}
      7 \\ 6 \\ 2
    \end{array}
    \right).
\end{equation*}

The vector $l$ describes the \emph{lengths} of the $0$-edges. This in
turn will be used to find the external angles at the $0$-vertices ($=$
postcritical points). More precisely for each $p\in \post(f)$ we will
define an external angle $\theta(p)$. 

\smallskip
The $0$-edges $E_j\subset \CC$ inherit the orientation of $\CC$. Let
$p_0, \dots, p_{k-1}$ be the postcritical points, labeled in
positively cyclical order on $\CC$, such that $p_0$ is the initial
point of (the first) $0$-edges $E_0$. 

Assume first that $p_0$ is a
fixed point of $f$ (as in the example from Figure~\ref{fig:2}). 
Then we set the external angle of $p_0$ equal to $0$, i.e.,
$\theta(p_0)=0$. The external angle of the other postcritical points
is now given by $\theta(p_1)=l_0$, $\theta(p_2)=l_0+l_1, \dots,
\theta(p_j)=l_0 + \dots + l_{j-1}$. Thus the difference between the
external angles of $p_j$ and $p_{j+1}$ is always given by $l_j$ (here
indices are taken $\bmod\, k$).  

Assume now that $p_0$ is mapped by $f$ to $p_j$. Let $l(p_0, p_j):= l_0
+ \dots + l_{j-1}$, i.e., the total length of all $0$-edges between
$p_0$ and $p_j$ (in the positive direction on $\CC$). We now desire
that $d\,\theta(p_0)= \theta(p_j) = \theta(p_0) + l(p_0, p_j)$. Thus we
define
\begin{equation*}
  \theta(p_0)= \frac{l(p_0,p_j)}{d-1},
\end{equation*}
and $\theta(p_1)= \theta(p_0) + l_0, \dots, \theta(p_i)= \theta(p_0) +
l_0 + \dots + l_{i-1}$. We will however only need $\theta(p_0)$ in the
following. 

\smallskip
We now assign external angles at the critical points. More precisely
we want to find the external angles at critical points of polynomials
$P_\wt, P_\bt$ (called the white/black polynomials). These are the
polynomials into which $f$ is unmated, i.e., $f$ will be
(topologically equivalent to) the mating of $P_\wt, P_\bt$. The
external angles will give the critical portraits of the polynomials
$P_\wt, P_\bt$. 

First we define
the length of a $1$-edge $E^1$ of type $j$ by $l(E^1)= l_j/d$. Note
that there are $d$ $1$-edges of each type, thus $\sum l(E^1)=1$ (where
the sum is taken over all $1$-edges). 

Denote by $\E^1$ the set of all $1$-edges. Since $H$ deforms
$\CC$ to $\CC^1= \bigcup \E^1$ it follows that the orientation of
$\CC$ together with $H$ induces a cyclical ordering $E^1_0, \dots ,
E^1_{kd -1}$ on the $1$-edges, as well as an orientation on each
$1$-edge. Here we start 
the labeling at $p_0$, i.e., the initial point of $E^1_0$ is
$p_0$. Note that the type of $1$-edges in this cyclical ordering is
changing cyclically, i.e., if $E^1_i$ is of type $j$, then $E^1_{i+1}$
is of type $j+1$ (here the lower index is taken $\bmod \,kd$, the type
is taken $\bmod \,k$). 

Assume the $1$-edge $E^1_j$ ends at the critical point $c$. Then an
external angle associated with $c$ is 
$$\theta(E^1_j):=\theta(p_0) + l(E^1_0) + \dots + l(E^1_{j}).$$
There are $\deg_f(c)$ such (oriented) $1$-edges
ending at $c$, hence different external angles associated to $c$. We
will have to 
decide, which belong to critical points of the white polynomial, and
which belong to critical points of the black polynomial. However, the
situation is more complicated: $c$ might be associated to several
distinct critical points of the white polynomial, as well as several
distinct critical points of the black polynomial. 

\smallskip
Recall that the set $\W$ was the set of preimages of the $1$-vertices
by $H_1$ located on $\CC$ (see
Definition~\ref{def:elementary_pseudo}\eqref{item:H2}). The points in
$\W$ divide $\CC$ into closed arcs, each of which is mapped by $H_1$
homeomorphically to a $1$-edge by $H_1$ (see
Definition~\ref{def:elementary_pseudo}~\eqref{item:H3}). Since there
are as many such arcs as $1$-edges (i.e., $kd$), there are $kd$ points
in $\W$. Let the points in $\W$ be $w_0,\dots, w_{kd-1}$ labeled
positively on $\CC$, such that $w_0=p_0$.

Consider $[c]:= (H_1)^{-1}(c)$, i.e., the set of points that are
deformed by $H$ to $c$. This set contains $\deg_f(c)$ points of $\W=
\{w_j\}\subset \CC$ (see Remark~\ref{rem:preH}). 
Recall from Lemma~\ref{lem:pseudo-equiv} that
this set is connected. To clarify: the set of points in $\CC$ that is
deformed by $H$ to $c$ (i.e., $[c]\cap \CC$) is finite, while the set
of \emph{all} points in $\CDach$ that is deformed by $H$ to $c$ (i.e.,
$[c]$) is infinite. 
 
Consider the components of $[c]\setminus
\CC$. Such a component $C$ is called white/black if it is contained in the
white/black $0$-tile (i.e., in $X^0_\wt$ or $X^0_\bt$). Furthermore we
call $C$ non-trivial if it contains at least two distinct points $w_i,
w_j\in \W$ in its boundary. Each non-trivial white component $C$ corresponds to
a critical point of the white polynomial, each non-trivial black
component $C$ corresponds to a critical point of the black polynomial.  

\smallskip
Let $C$ be a white non-trivial component. The set of external angles
associated to $C$ is now the set
\begin{equation*}
  \{\theta(E^1_j) \mid \text{$w_j\in \W$ contained in the closure of
    $C$}\}. 
\end{equation*}
The list of all these sets of external angles (for all critical points
$c$ and all white components of $[c]\setminus \CC$) forms the critical
portrait of the white polynomial $P_\wt$. 

The white critical portrait for the pseudo-equator indicated in
Figure~\ref{fig:2} is
\begin{equation*}
  \left\{\frac{7}{60}, \frac{22}{60}, \frac{37}{60} \right\},
  \left\{\frac{43}{60}, \frac{58}{60}\right\}.
\end{equation*}

\smallskip
The critical portrait of the black polynomial is constructed in almost
the same fashion. There is a slight difference however. This
difference appears, since in
the construction of mating a point with external angle $\theta$ in the
Julia set of one polynomial is identified with a point with external
angle $-\theta$ from the Julia set of the other polynomial. 

Thus let $C$ be a non-trivial black component of $[c]\setminus \CC$
(for some critical point $c$). Then the set of external angles
associated to $C$ is
\begin{equation*}
  \{1-\theta(E^1_j) \mid \text{$w_j\in \W$ contained in the closure of
    $C$}\}. 
\end{equation*}
The list of all these sets (for all critical points $c$ and all
non-trivial black components of $[c]\setminus \CC$) form the critical
portrait of the black polynomial $P_\bt$.

\smallskip
The black critical portrait of the pseudo-equator in
Figure~\ref{fig:2} is
\begin{equation*}
  \left\{\frac{15}{60}, \frac{30}{60}, \frac{45}{60}\right\},
  \left\{\frac{2}{60}, \frac{47}{60}\right\}.
\end{equation*}

Thus the critical portraits of the white and black polynomials can be
read off from the pseudo-isotopy in an elementary combinatorial
way. These determine the white and black polynomials $P_\wt, P_\bt$
uniquely. In \cite{inv_Peano} and \cite{exp_quotients} it is shown
that $P_\wt \mate P_\bt$ is topologically conjugate to $f$. Thus the
existence of a pseudo-equator not only shows that $f$ arises as a
mating, but allows to recover the polynomials into which $f$ unmates
in an elementary fashion. 

\section{Examples of unmatings}
\label{sec:examples-unmatings}

Here we show several examples of unmatings. 

\smallskip
We first show that shared matings are ubiquitous. We list all shared
matings of the example from Section~\ref{sec:an-example} that can be
found with the sufficient condition from
Section~\ref{sec:anoth-suff-crit}. 

\begin{figure}
  \centering
    \begin{overpic}
      [width=10cm, 
      tics =20]{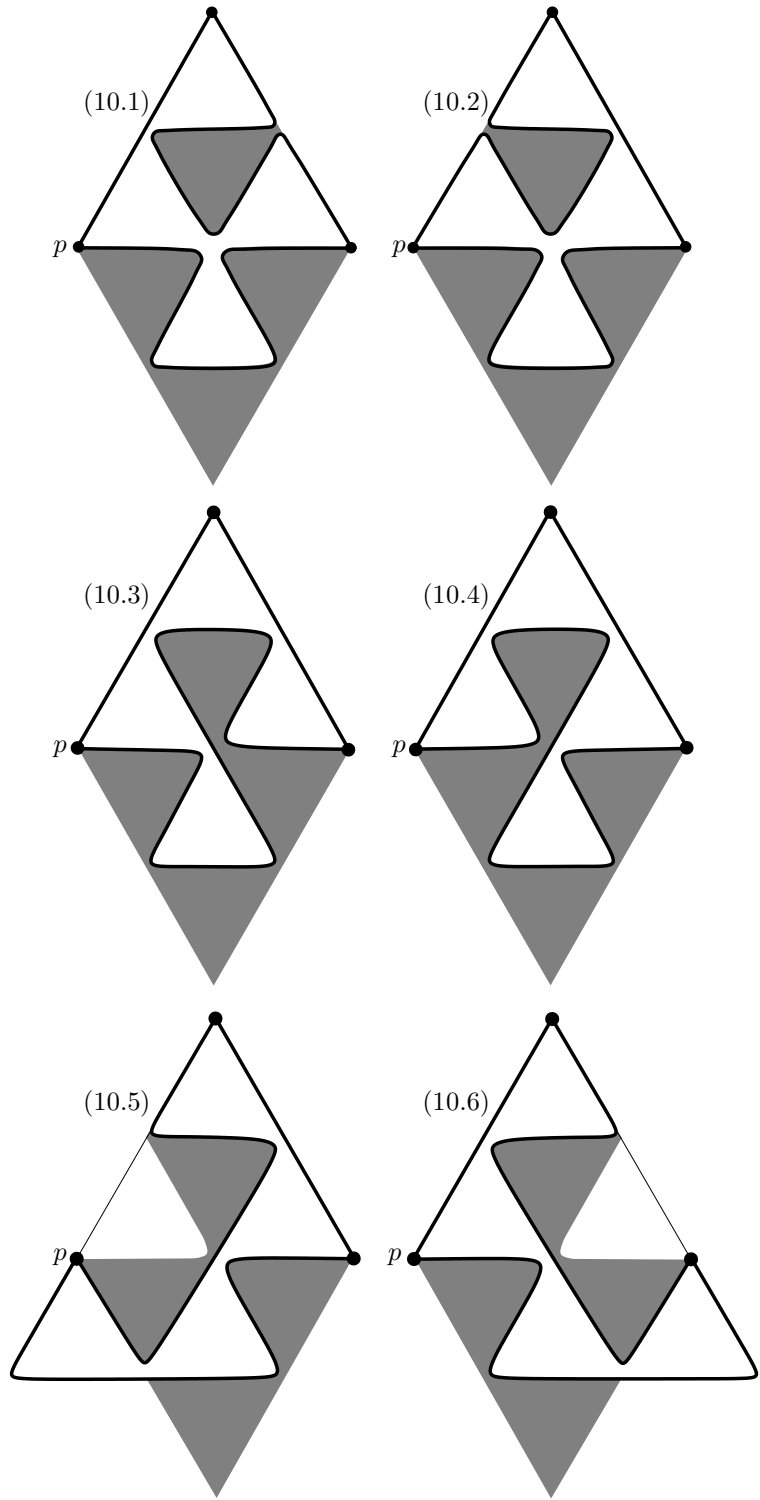}
      \put(5,93){\eqref{eq:unmate1}}
      \put(3,83.5){$p$}
      \put(27.7,93){\eqref{eq:unmate2}}
      \put(25.7,83.5){$p$}
      \put(5,60){\eqref{eq:unmate3}}
      \put(3,50){$p$}
      \put(27.7,60){\eqref{eq:unmate4}}
      \put(25.7,50){$p$}
      \put(5,26){\eqref{eq:unmate5}}
      \put(3,16){$p$}
      \put(27.7,26){\eqref{eq:unmate6}}
      \put(25.4,16){$p$}
    \end{overpic}
    \caption{Different pseudo-equators.}
    \label{fig:pseudo-equators}
\end{figure}

The critical portraits of the polynomials into which the map $f$
unmates are as follows. They are obtained from the pseudo-equators
shown in Figure~\ref{fig:pseudo-equators}. 
It is convenient to write the common denominator
of the angles of a critical portrait outside the parentheses, i.e.,
we write $\frac{1}{N}\{n,m\},\{i,j\}$ for 
$\{n/N, m/N\}, \{i/N, j/N\}$ and so on. By $\wt$ we denote 
the critical portrait of the white polynomial, by $\bt$ the critical 
portrait of the (corresponding) black polynomial into which $f$
unmates. 

\begin{align}
  \label{eq:unmate1}
  &\wt: \frac{1}{60} \{7,22,37\}, \{43, 58\} 
  &\bt&: \frac{1}{60} \{2,47\}, \{15, 30, 45\}
  \\
  \label{eq:unmate2}
  &\wt: \frac{1}{20} \{2,7,17\},\{10, 15\} 
  &\bt&: \frac{1}{20} \{2,7, 17\}, \{10, 15\}  
  \\
  \label{eq:unmate3}
  &\wt: \frac{1}{60} \{11,26\}, \{29,59\}, \{30, 45\} 
  &\bt&:\frac{1}{60} \{1,46\}, \{19, 34\}, \{15, 45\}  
  \\
  \label{eq:unmate4}
  &\wt: \frac{1}{60} \{14,59\}, \{26,41\}, \{15, 45\} 
  &\bt&: \frac{1}{60} \{1,31\}, \{15, 30\},\{34, 49\}  
  \\
  \label{eq:unmate5}
  &\wt: \frac{1}{20} \{3,18\}, \{7,17\}, \{10, 15\} 
  &\bt&: \frac{1}{20} \{3,18\},  \{7, 17\}, \{10, 15\}  
  \\  
  \label{eq:unmate6}
  &\wt: \frac{1}{60} \{7,37\}, \{15,30\}, \{43, 58\} , 
  &\bt&: \frac{1}{60} \{2,47\}, \{15, 45\}, \{23, 38\}  
\end{align}

The pseudo-equators in Figure~\ref{fig:pseudo-equators} are not all
possible, all others however are obtained from these by rotation (by
$2\pi/3$ and $4\pi/3$). In the above the point labeled $p$ in
Figure~\ref{fig:pseudo-equators} is the point $-1$. We now rotate
each pseudo-equator by $2\pi/3$, meaning the point labeled
$p$ is $1$. From these new pseudo-equators we obtain unmatings of $f$
into polynomials with the following critical portraits.

\begin{align}
  \label{eq:unmate7}
  \tag{\ref{eq:unmate1}'}
  &\wt: \frac{1}{60} \{2,47\}, \{15, 30, 45\}
  &\bt&: \frac{1}{60} \{7,22,37\}, \{43, 58\} 
  \\
  \label{eq:unmate8}
  \tag{\ref{eq:unmate2}'}
  &\wt: \frac{1}{60} \{13,58\},\{15,30,45\} 
  &\bt&: \frac{1}{60} \{2, 17\}, \{23, 38, 53\}  
  \\
  \label{eq:unmate9}
  \tag{\ref{eq:unmate3}'}
  &\wt: \frac{1}{20} \{1,11\}, \{5,10\}, \{14, 19\} 
  &\bt&:\frac{1}{20} \{1,11\}, \{5, 10\}, \{14, 19\}  
  \\
  \label{eq:unmate10}
  \tag{\ref{eq:unmate4}'}
  &\wt: \frac{1}{20} \{1,6\}, \{9,19\}, \{10, 15\} 
  &\bt&:\frac{1}{20} \{1,6\}, \{9, 19\}, \{10, 15\}  
  \\
  \label{eq:unmate11}
  \tag{\ref{eq:unmate5}'}
  &\wt: \frac{1}{60} \{13,58\}, \{15,45\}, \{22, 37\} 
  &\bt&:\frac{1}{60} \{2,17\}, \{23, 53\}, \{30, 45\}  
  \\
  \label{eq:unmate12}
  \tag{\ref{eq:unmate6}'}
  &\wt: \frac{1}{60} \{2,47\}, \{15,45\}, \{23, 38\} 
  &\bt&:\frac{1}{60} \{7,37\}, \{15, 30\}, \{43, 58\}  
\end{align}

Finally we rotate the pseudo-equators in
Figure~\ref{fig:pseudo-equators} by $4\pi/3$, meaning that the point
labeled $p$ is $\infty$. We obtain unmatings of the map $f$ into
polynomials with the following critical portraits. 

\begin{align}
  \label{eq:unmate13}
  \tag{\ref{eq:unmate1}''}
  &\wt: \frac{1}{20} \{3,13, 18\}, \{5, 10\}  
  &\bt&: \frac{1}{20} \{3,13, 18\}, \{5, 10\}  
  \\
  \label{eq:unmate14}
  \tag{\ref{eq:unmate2}''}
  &\wt: \frac{1}{60} \{2,17\},\{23,38,53\} 
  &\bt&: \frac{1}{60} \{13, 58\}, \{15, 30, 45\}  
  \\
  \label{eq:unmate15}
  \tag{\ref{eq:unmate3}''}
  &\wt: \frac{1}{60} \{1,46\}, \{15,45\}, \{19,34\} 
  &\bt&:\frac{1}{60}\{11,26\}, \{29, 59\}, \{30, 45\}  
  \\
  \label{eq:unmate16}
  \tag{\ref{eq:unmate4}''}
  &\wt: \frac{1}{60} \{1,31\}, \{15,30\}, \{34, 49\} 
  &\bt&:\frac{1}{60}\{14,59\}, \{15, 45\}, \{26, 41\}  
  \\
  \label{eq:unmate17}
  \tag{\ref{eq:unmate5}''}
  &\wt: \frac{1}{60} \{2,17\}, \{23,53\}, \{30, 45\} 
  &\bt&:\frac{1}{60}\{13,58\}, \{15, 45\}, \{22, 37\}  
  \\
  \label{eq:unmate18}
  \tag{\ref{eq:unmate6}''}
  &\wt: \frac{1}{20} \{2,17\}, \{3,13\}, \{5, 10\} 
  &\bt&:\frac{1}{20}\{2,17\}, \{3, 13\}, \{5, 10\}  
\end{align}


Here are some observations from these examples. Several, namely
\eqref{eq:unmate2}, \eqref{eq:unmate5}, \eqref{eq:unmate9},
\eqref{eq:unmate10}, \eqref{eq:unmate13}, \eqref{eq:unmate18}, are
obtained by mating the same white polynomial to the same black
polynomial. 

Somewhat more interesting (and possibly surprising) is the following
phenomenon. Consider the pairs \eqref{eq:unmate1} and
\eqref{eq:unmate7}, \eqref{eq:unmate6} and \eqref{eq:unmate12},
\eqref{eq:unmate3} and \eqref{eq:unmate15}, \eqref{eq:unmate4} and
\eqref{eq:unmate16}, as well as \eqref{eq:unmate8} and
\eqref{eq:unmate14}. In each of these cases the map $f$ is unmated
into the same two polynomials. The pseudo-equators however by which
these unmatings are achieved are distinct. Thus the corresponding
invariant Peano curves are distinct. Put differently, $f$ can be
obtained as the mating of two polynomials $P_\wt, P_\bt$ in two
distinct ways. Note that the role of the black and white polynomials
however is interchanged in each case. 

\smallskip
We next consider some unmatings of the map $g\colon \CDach \to \CDach$
given by
\begin{equation}
\label{eq:defg}
  g(z) = 1 + \frac{\omega-1}{z^3},
\end{equation}
where $\omega= \exp(4\pi i/3)$. Note that $g(z) = \tau(z^3)$, where
$\tau(w)= 1 + (\omega-1)/w$ is a M\"obius transformation that maps the
upper half plane to the half plane above the line through $\omega, 1$
(more precisely $\tau: 0\mapsto \infty , 1 \mapsto \omega,
\infty\mapsto 1$). The critical points of $g$ are $0, \infty$, which
are mapped as follows.

\begin{equation*}
  \xymatrix{
    0 \ar[r]^{3:1} & \infty \ar[r]^{3:1} & 1 \ar[r] & \omega \ar@(r,u)[]
  }
\end{equation*}

Thus $\post(g)=\{\infty,1,\omega\}$, i.e., $g$ is postcritically
finite. Since each critical point is strictly preperiodic, it follows
that the Julia set of $g$ is all of $\CDach$. The crucial property for
our purposes is that $\infty$ is at the same time a critical as well
as a postcritical point. 

As $\CC\supset \post(g)$ we choose the extended line through
$\omega,1$, i.e., the circle on $\CDach$ through $\omega, 1, \infty$,
oriented positively as boundary of the half plane above this
line. It holds
\begin{equation*}
  g^{-1}(\CC)= \bigcup_{j=0, \dots, 5} R_j,
\end{equation*}
where $R_j= \{r \exp(2\pi i j/6) \mid 0\leq r\leq \infty\}$. 

\begin{figure}
  \centering
  \begin{overpic}
    [width=10cm, 
    tics =20]{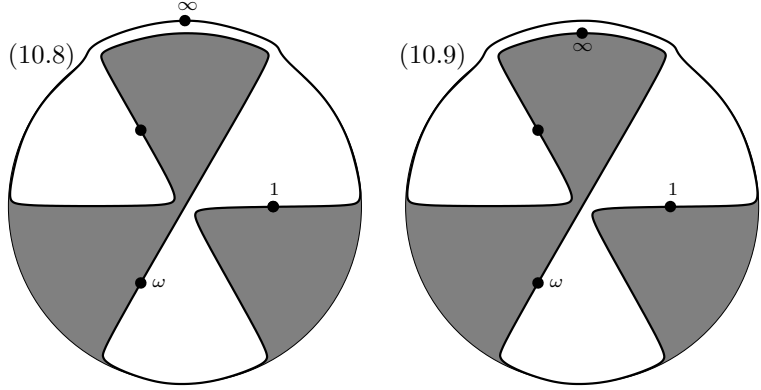}
    \put(75,44.5){${\scriptstyle \infty}$}
    \put(22.5,50){${\scriptstyle \infty}$}
    \put(19.2,13.1){${\scriptstyle \omega}$}
    \put(72,13.1){${\scriptstyle \omega}$}
    \put(34.7,25.4){$\scriptstyle 1$}
    \put(87.7,25.4){$\scriptstyle 1$}
    \put(0,43){\eqref{eq:eq1_g}}
    \put(52,43){\eqref{eq:eq2_g}}
  \end{overpic}
  \caption{Shared invariant Peano curve.}
  \label{fig:6}
\end{figure}
 
We show two pseudo-equators for $g$ in Figure~\ref{fig:6}. In both
cases the points on the circle all represent the point $\infty$. The
narrow white channel indicates that two white $1$-tiles are connected
at $\infty$ (in both cases). The
connection of $1$-tiles is the same in \eqref{eq:eq1_g} and
\eqref{eq:eq2_g}, with one exception: the marking of the connection at
$\infty$ (more precisely the cnc-partition representing the connection
at $\infty$) differ. This means that in the geometric representation
of Figure~\ref{fig:6} the point $\infty$ is in different
positions. From these pseudo-equators we obtain unmatings of $g$ into
polynomials with the following critical portraits. 
\begin{align}
  \label{eq:eq1_g}
  &\wt: \frac{1}{27} \{4, 22\}, \{12,21\}  
  &\bt: \frac{1}{27} \{5,14\}, \{15,24\} 
  \\
  \label{eq:eq2_g}
  &\wt: \frac{1}{27} \{7, 25\}, \{12,21\}  
  &\bt: \frac{1}{27} \{2,11\}, \{15,24\} 
\end{align}

\begin{remark}
  At first the author thought that this example would result in two
  distinct pairs of polynomials that would do not only yield the same
  rational map when mated, but also the same invariant Peano
  curve. This however is not true. We do not give the precise argument
  here. This can however be found in \cite[Section 8]{inv_Peano}. 
\end{remark}

\section{A mating not arising from a pseudo-equator}
\label{sec:mating-not-arising}

In this section we give an example of a mating that does not arise via
a pseudo-equator. Thus we show that the sufficient condition that a
rational map arises as a mating given in Theorem~\ref{thm:mating_suff}
is not necessary. 

\smallskip
More precisely we will mate two postcritically finite polynomials and
show that the resulting rational map has no pseudo-equator. The Julia
set of this map will be the whole sphere.

\smallskip
In the construction of the topological mating a point in the white
Julia set where the external ray $R_\theta$ lands is identified with a
point in the black Julia set where the external ray $R_{-\theta}$
lands. To make the following more readable we define for any $\theta
\in [0,1)$ $\theta^*:= 1- \theta$. Thus in the topological mating the
landing point of $R_\theta$ in the white Julia set is identified with
the landing point of $R_{\theta^*}$ in the black polynomial. 

\smallskip
The white polynomial will be $P_\wt:= z^2 +i$. It is well-known that the
critical portrait of $P_\wt$ is $\{1/12, 7/12\}$. The external angles
are mapped under the angle doubling map in the following way:

\begin{equation*}
  \xymatrix {
    (1/12), (7/12) \ar[r] & (1/6) \ar[r]  & (1/3) \ar[r] &
    (2/3)\ar@/_1pc/[l]
  }.
\end{equation*}

Furthermore at the
$\alpha$-fixed point of $P_\wt$ the external rays with angles $\{1/7,
2/7, 4/7\}$ land. 

As the black polynomial $P_\bt$ we pick the
(quadratic, monic, postcritically finite) polynomial with the critical
portrait $\{(1/28)^*, (15/28)^*\}$. Thus $P_\bt= z^2+ c_\bt$, where $c_\bt$ is
the landing point of the external ray $R_{(1/14)^*}$ in the Mandelbrot
set. The external angles are mapped under angle doubling in the
following way:

\begin{equation*}
  \xymatrix {
    (1/28)^*, (15/28)^* \ar[r] & (1/14)^* \ar[r]  & (1/7)^* \ar[r] &
    (2/7)^* \ar[r] & (4/7)^* \ar@/_1.5pc/[ll]
  }.
\end{equation*}

\begin{lemma}
  \label{lem:matepwb}
  The topological mating of $P_\wt, P_\bt$ exists and is topologically
  conjugate to a rational map $f\colon \CDach \to \CDach$ of degree
  $2$. The two critical points $a,b$ of $f$ are mapped by $f$ in the
  following way:
  \begin{align*}
    &\xymatrix{
      a \ar[r]^{2:1} & p_1 \ar[r] & p_2 \ar@(r,u)[]
    }
    \\
    &\xymatrix{
      b \ar[r]^{2:1} & q_1 \ar[r] & q_2 \ar[r] & q_3 \ar@/_1pc/[l]
    }.
  \end{align*}
\end{lemma}
Thus $f$ has $5$ postcritical points, namely $p_1,p_2, q_1, q_2,
q_3$. The Julia set of $f$ is the whole sphere, since each critical
point is strictly preperiodic.

\begin{proof}
  It is well-known that $i$, i.e., the critical value of $P_\wt$ is in
  the $1/3$-limb of the Mandelbrot set. Indeed the external angle at
  the critical value $i$ is $1/6$. The external rays $R_{1/7}$
  and $R_{2/7}$ disconnect the $1/3$-limb from the main cardioid. Since
  $1/7< 1/6 < 2/7$ it follows that $i$ lies in this limb. 

  The conjugate limb is the $2/3$-limb, it is disconnected from the
  main cardioid by the external rays $R_{5/7}, R_{6/7}$. Since
  $(1/14)^*= (13/14)> 6/7$, it follows that $i, c_\bt$ do not lie in
  conjugate limbs of the Mandelbrot set. From the Rees-Shishikura-Tan
  theorem (Theorem~\ref{thm:matingdeg2}) it follows that the
  topological mating of $P_\wt, P_\bt$ exists and is topologically
  conjugate to a rational map (which is postcritically finite and of
  degree $2$).   

  \smallskip
  Clearly the three postcritical points of $P_\bt$ with external angles
  $(1/7)^*, (2/7)^*$, $(4/7)^*$ are all identified with the
  $\alpha$-fixed point of $P_\wt$ under the topological mating, i.e.,
  correspond to a single postcritical point in the mating. 
  It remains to show that there are not more identifications of
  postcritical points. 

  \smallskip
  Consider first the external angle $1/3$. Let $z_{(1/3)^*}$ be the
  landing point of $R_{(1/3)^*}$ in the Julia set $\J_\bt$ of
  $P_\bt$. We want to show that $R_{(1/3)^*}$ is the only external
  ray landing at $z_{(1/3)^*}$. Indeed $(1/3)^*$ is fixed under
  the second iterate of the angle-doubling map. Thus if $R_{\theta^*}$
  lands at $z_{(1/3)^*}$ it has to be fixed under the second iterate
  of the angle doubling map as well. Thus either $\theta^*= (2/3)^*$
  or $\theta^* = 0$. In the first case the external rays $R_{1/3},
  R_{2/3}$ in the dynamical plane of $P_\wt$ together with the
  external rays $R_{(1/3)^*}, R_{(2/3)^*}$ disconnect the formal
  mating sphere. This cannot happen by the Rees-Shishikura-Tan
  theorem. The second case cannot happen, since for quadratic
  polynomials the landing point of $R_0$, i.e., the $\beta$-fixed
  point, is not the landing point of any other external ray.  

  By the same argument it follows that the landing point of $R_{(2/3)^*}$
  (in the dynamical plane of $P_\bt$) is not the landing point of any
  other external ray. 

  \smallskip
  Consider now the landing point $z_{(1/6)^*}$ of $R_{(1/6)^*}$, in
  the dynamical plane of $P_\bt$. This point is not a critical point
  which is mapped to $z_{(1/3)^*}$ by $P_\bt$. Since by the above there is a
  single external ray landing at $z_{(1/3)^*}$ it follows that there is
  a single external ray landing at $z_{(1/6)^*}$. 

  Thus it follows that the three postcritical points of $P_\wt$
  descend to three distinct postcritical points $q_1,q_2,q_3$ of $f$
  and their dynamics is as described in the statement. 

  \medskip
  It remains to show that the postcritical point of $P_\bt$ at which
  the external ray $R_{(1/14)^*}$ (in the dynamical plane of $P_\bt$)
  lands is not identified with other postcritical or critical points. 

  \smallskip
  We first note that at the three postcritical points of $P_\bt$ at
  which the external rays $R_{(1/7)^*}, R_{(2/7)^*}, R_{(4/7)^*}$ land
  no other external ray lands.  Indeed any other such external ray
  would be a fixed point of the third iterate of the angle doubling
  map, i.e., has to be an external ray of the form
  $R_{(k/7)^*}$, where $k=3,5,6$. 
  Assume first that $R_{(1/7)^*}, R_{(6/7)^*}$ land at the same
  point. This is impossible, since the external rays $R_{(1/28)^*},
  R_{(15/28)^*}$ (which both land at the critical point of $P_\bt$)
  are not contained in one component of $\C\setminus (R_{(1/7)^*} \cup
  R_{(6/7)^*}$. The same argument shows that $R_{(1/7)^*},
  R_{(5/7)^*}$ cannot land at the same point. Finally assume that
  $R_{(1/7)^*}, R_{(3/7)^*}$ land at the same point. Mapping these
  external rays by $P_\bt$ yields that $R_{(2/7)^*}, R_{(6/7)^*}$ land
  at the same point. Again this is impossible, since $R_{(1/28)^*}
  \cup R_{(15/28)^*}$ do not lie in the same component of $\C\setminus
  (R_{(2/7)^*}\cup R_{(6/7)^*})$. 

  \smallskip
  Consider the landing point $z_{1/14}$ of the external ray
  $R_{1/14}$ in the dynamical plane of $P_\wt$. This is the preimage
  of the $\alpha$-fixed point of $P_\wt$. Thus the external rays
  $R_{9/14}, R_{11/14}$ (in the dynamical plane of $P_\wt$) land at
  $z_{1/14}$ as well.  

  Consider now one the corresponding external rays $R_{(9/14)^*},
  R_{(11/14)^*}$ in the dynamical plane of $R_\bt$. These are mapped
  by $P_\bt$ to $R_{(2/7)^*}, R_{(4/7)^*}$. By the above there is no
  other external ray landing at these endpoints. Thus there are no
  other external rays landing at the same points as $R_{(9/14)^*},
  R_{(11/14)^*}$. 

  \smallskip
  Thus the $4$ postcritical points of $P_\bt$ descend to $2$
  postcritical points $p_1, p_2$ of $f$, which are mapped as stated. 
\end{proof}

\begin{theorem}
  \label{thm:fnomating}
  The rational map $f\colon \CDach \to \CDach$ from
  Lemma~\ref{lem:matepwb}, which arises as the mating of $P_\wt$ and
  $P_\bt$, does not have a pseudo-equator (as in
  Definition~\ref{def:elementary_pseudo}).  
\end{theorem}

\begin{proof}
  To show that no
  pseudo-equator as in Definition~\ref{def:elementary_pseudo} exists,
  we will consider several cases. Namely we consider the different
  cyclical orders in which $\CC$ may traverse the postcritical
  points. Assume a pseudo-isotopy $H\colon \CDach \times [0,1]\to
  \CDach$ deforms $\CC$ to $\CC^1= f^{-1}(\CC)$ as in
  Definition~\ref{def:elementary_pseudo}. Then $H$ induces an order in
  which $\CC^1$ is traversed. The cyclical order in which $\CC^1$
  traverses the postcritical points has to agree with the cyclical
  order of the postcritical points on $\CC$, since $H$ is a
  pseudo-isotopy rel.\ $\post(f)$. We will show that this is
  impossible. 

  \smallskip
  Thus let $\CC\subset \CDach$ be a Jordan curve with
  $\post(f)\subset \CC$, which is (arbitrarily) oriented. The white
  and black $0$-tiles $X^0_{\wt}, X^0_\bt$, as well as the white and
  black $1$-tiles are defined in terms of $\CC$ as usual. There are
  two white (as well as two black) $1$-tiles, which intersect at the
  critical points $a,b$. To obtain a connection of white $1$-tiles
  that yields a spanning tree, we have to connect the two white
  $1$-tiles either at $a$ or at $b$. Recall that each white $1$-tile
  is mapped homeomorphically to $X^0_\wt$ by $f$. Thus the two
  preimages of any point $p_2, q_2, q_3$, i.e., the postcritical points
  which are not critical values, have to lie in distinct white
  $1$-tiles.

  \smallskip
  Let $q_2'$ be the preimage of $q_3$ by $f$ distinct from
  $q_2$. Augmenting this point to the diagram in
  Lemma~\ref{lem:matepwb} we obtain that the points $a,b, p_1, p_2,
  q_1,q_2, q_2', q_3$ are mapped by $f$ as follows:
  \begin{align*}
    &\xymatrix{
      a \ar[r]^{2:1} & p_1 \ar[r] & p_2 \ar@(r,u)[]
    }
    \\
    &\xymatrix{
      b \ar[r]^{2:1} & q_1 \ar[r] & q_2 \ar[r] & q_3 \ar@/_1pc/[l] &
      q_2' \ar[l]
    }.
  \end{align*}
  Note that the set $\{a,b, p_1, p_2, q_1,q_2, q_2', q_3\}$ contains
  all preimages of the postcritical points. 

\begin{figure}
  \centering
  \begin{overpic}
    [width=12cm, 
    tics=20]{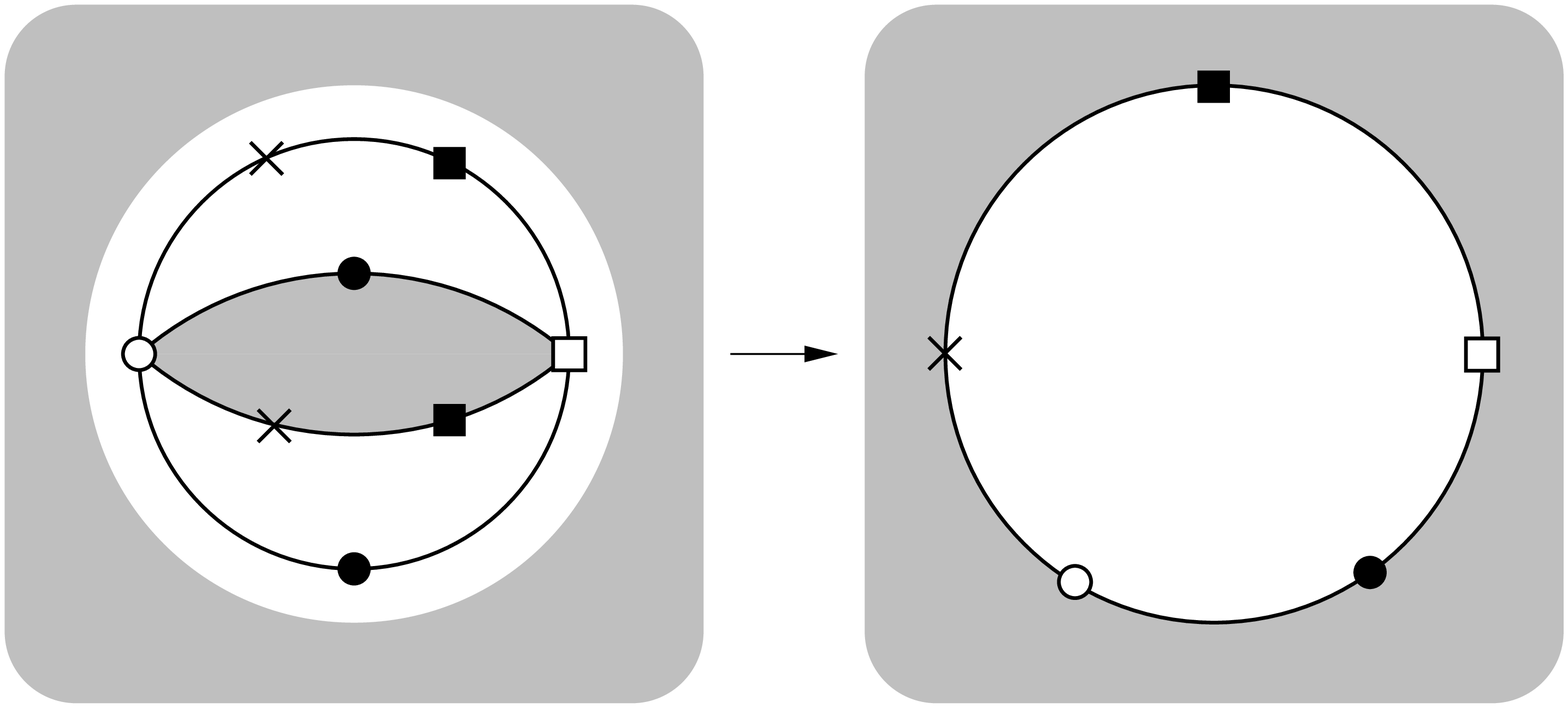}
    \put(38,22){${\scriptstyle b}$} 
    \put(5.8,22){${\scriptstyle a}$} 
    \put(21,6){${\scriptstyle p_2}$}
    \put(14,14.5){${\scriptstyle q_2\vee q_2'}$}
    \put(25.5,15){${\scriptstyle q_3\vee q_1}$}
    \put(14.5,36.7){${\scriptstyle q_2'\vee q_2}$}  
    \put(25.5,36.5){${\scriptstyle q_1\vee q_3}$}
    \put(22,29.5){${\scriptstyle p_1}$} 
    %
    \put(68,10.5){${\scriptstyle p_1}$} 
    \put(85,10.5){${\scriptstyle p_2}$} 
    \put(62,22){${\scriptstyle q_3}$} 
    \put(91,22){${\scriptstyle q_1}$} 
    \put(77,36.5){${\scriptstyle q_2}$} 
    \put(49,24){${ f}$} 
  \end{overpic}
  \caption{Case 1.}
  \label{fig:4}
\end{figure}

\begin{figure}
  \centering
  \begin{overpic}
    [width=12cm, 
    tics=20]{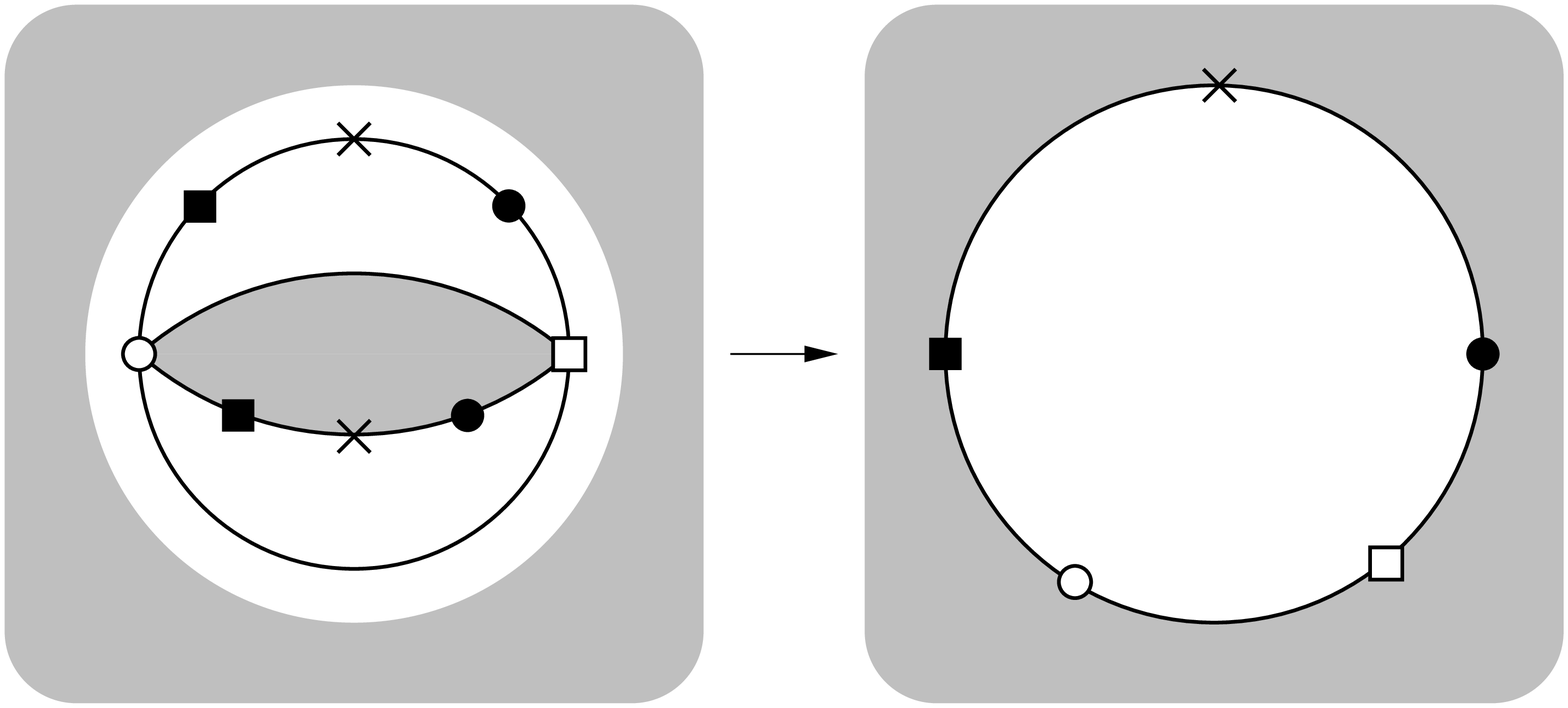}
    %
    \put(68,10.5){${\scriptstyle p_1}$} 
    \put(85,10.5){${\scriptstyle q_1}$} 
    \put(62,22){${\scriptstyle q_2}$} 
    \put(91,22){${\scriptstyle p_2}$} 
    \put(77,36.5){${\scriptstyle q_3}$} 
    \put(49,24){${ f}$} 
    %
    \put(38,22){${\scriptstyle b}$} 
    \put(5.8,22){${\scriptstyle a}$} 
    \put(29,15.7){${\scriptstyle p_2}$}
    \put(29.7,29.7){${\scriptstyle p_1}$}
    \put(14,30){${\scriptstyle q_1}$}
    \put(14,15.7){${\scriptstyle q_3}$}
  \end{overpic}
  \caption{Case 3.}
  \label{fig:5}
\end{figure}

  \smallskip
  \begin{case1}[$p_1,p_2$ are adjacent on $\CC$]
    
    \mbox{}
    
    This means that one arc on $\CC$ between $p_1,p_2$ does
    not contain any other postcritical point. One such situation is
    pictured in Figure~\ref{fig:4}. Assume the two white $1$-tiles are
    connected at $b$. Traverse $\CC^1:= f^{-1}(\CC)$ in the order
    induced by this connection. Note that on each of the two paths (on
    $\CC^1$) between $p_1, p_2$ there is at least one other
    postcritical point. Thus there is no pseudo-isotopy rel.\
    $\post(f)$ that deforms $\CC$ to $\CC^1$. The same argument works
    in the case when the white $1$-tiles are connected at $a$. Indeed
    the same argument always works when $p_1,p_2$ are adjacent on
    $\CC$, i.e., Case~1 is impossible.  
  \end{case1}
  
  By exactly the same argument as above we can rule out that $q_1,
  q_3$ are adjacent on $\CC$. Indeed as before they have to be in
  distinct white $1$-tiles, since they are both preimages of $q_2$. If
  the white $1$-tiles are connected either at $a$ or at $b$, there
  will be other postcritical points between $q_1, q_3$ when we
  traverse $\CC^1$.

  \begin{case2}[$p_2,q_2$ are adjacent on $\CC$]
    
    \mbox{}

    We choose the orientation on $\CC$ such that $q_2$ succeeds $p_2$
    on $\CC$. Consider the white $1$-tile $X$ containing $p_2$ in its
    boundary $\partial X$. Then $p_2$ is succeeded on $\partial X$,
    hence on $\CC^1$, by a $1$-vertex which is mapped to $q_2$ by
    $f$. This $1$-vertex must be $q_1$ or $q_3$. Hence on $\CC^1$ the
    postcritical point $p_2$ is succeeded by either $q_1$ or
    $q_3$. Hence $\CC^1$ is not obtained as a pseudo-isotopic
    deformation of $\CC$ rel.\ $\post(f)$ (in which case $p_2$ would
    be succeeded by $q_2$ on $\CC^1$). Hence Case 2 does not happen.
  \end{case2}
  
    From the above it follows that the cyclical order of the
    postcritical points on $\CC$ is either $p_1, q_1, p_2, q_3, q_2,
    p_1$ or $p_1, q_3, p_2, q_1, q_2, p_1$. Here we are using that we
    can choose the orientation on $\CC$ arbitrarily. 

    \begin{case3}
      [The cyclical order of the postcritical points on $\CC$ is $p_1, q_1, p_2, q_3, q_2,
    p_1$]

    \mbox{}
    
    The situation is illustrated in Figure~\ref{fig:5}. Since on $\CC$
    the point $p_1$ is succeeded by $q_1$ and $p_2$ is succeeded by $q_3$, the
    same has to hold on $\CC^1$. Note however, that on $\CC^1$ between
    $p_1, q_1$ and between $p_2, q_3$ there is a preimage of
    $q_3$. One of these preimages (marked with an {\textsf{x}} in the
    figure) of $q_3$ has to be $q_2$. Thus the cyclical ordering of
    the postcritical points on $\CC^1$ does not agree with the one on
    $\CC$. 
    \end{case3}

    \begin{case4}
      [The cyclical order of the postcritical points on $\CC$ is $p_1, q_3, p_2, q_1, q_2,
    p_1$]

    \mbox{}
    
    The situation is exactly analog to Case 3. In the right of
    Figure~\ref{fig:5} we have to interchange $q_1$ and $q_3$. Since
    $p_1$ is succeeded by $q_3$ and $p_2$ is succeeded by $q_1$ on
    $\CC$, the same has to hold in $\CC^1$. Thus in the right of
    Figure~\ref{fig:5} we have to interchange $q_1$ and $q_3$ for the
    situation at hand. One of the points marked by {\textsf x} has to
    be $q_2$, which gives a contradiction as before. 
    \end{case4}
\end{proof}

\section{Open questions}
\label{sec:open-questions}

Let $f\colon \CDach \to \CDach$ be a rational map.
Is the existence of an equator for $f$ \emph{sufficient} for $f$ to
arise as a mating? When $f$ is postcritically finite, this is widely
expected, but there does not seem to be a proof in the literature. 

\smallskip
A (pseudo-) equator $\EC$ has to be \emph{orientation-preserving}
(pseudo-) isotopic to $f^{-1}(\EC)$. \emph{Orientation-reversing}
(pseudo-) equators seem to be as common as orien\-ta\-tion-preserving
ones. From such a pseudo-equator it is possible to construct a
semi-conjugacy from $z^{-d}\colon S^1\to S^1$ to the map $f$ (here $d=
\deg f$) (in the case when $f$ is postcritically finite and
$\J(f)=\CDach$). There should be some sort of ``orientation-reversing 
mating'' associated to such orientation-reversing (pseudo-) equators. 

\smallskip
The form of Moore's theorem presented here allows one to shrink each 
ray-equivalence class to a point by a pseudo-isotopy $H$. Is it possible
to put further smoothness assumptions on $H$. In particular to \emph{embed
$H$ in a holomorphic motion}?

\bibliographystyle{alpha}
\bibliography{main}

\end{document}